\documentclass[reqno]{amsart}
\usepackage{amsfonts,amssymb,amsthm}
\usepackage{xcolor}
\usepackage{amsmath}
\usepackage{amsthm}
\usepackage{amsfonts}
\usepackage{latexsym}
\usepackage{verbatim}
\usepackage{amssymb}
\usepackage{upref}
\usepackage{mathrsfs}
\usepackage{enumerate}
\usepackage{amsmath}
\usepackage{tikz}\usetikzlibrary{patterns}
\usepackage{hyperref}
\title{Wave equations with  mass and dissipation}
\author{Wanderley Nunes do Nascimento \and Jens Wirth}
\address{Jens Wirth\\Institute of Analysis, Dynamics and Modelling\\Department of Mathematics\\University of Stuttgart\\70569 Stuttgart\\Germany}
\email{jens.wirth@mathematik.uni-stuttgart.de}

\address{Wanderley Nunes do Nascimento\\
Universidade Federal de S\~ao Carlos\\S\~ao Carlos\\ Brasil}
\email{wnunesmg@yahoo.com.br}

\newtheorem{thm}{Theorem}[section]
\newtheorem{lem}[thm]{Lemma}
\newtheorem{prop}[thm]{Proposition}

\theoremstyle{definition}
\newtheorem{hypo}{Hypothesis}
\theoremstyle{remark}
\newtheorem{rem}{Remark}[section]
\newtheorem{defn}{Definition}
\newtheorem{exam}[rem]{Example}
\numberwithin{equation}{section}

\def\R{\mathbb R}
\def\C{\mathbb C}

\def\d{\mathrm d}
\DeclareMathOperator{\diag}{diag}
\let\Re\relax
\DeclareMathOperator{\Re}{Re}
\DeclareMathOperator{\supp}{supp}
\def\D{\mathrm D}
\def\i{\mathrm i}
\def\e{\mathrm e}
\DeclareMathOperator*{\slim}{s-lim}
\begin{document}
\begin{abstract}
In this paper we consider a wave model 
with non-effective mass and dissipation terms and provide sharp descriptions of its representation of solutions.
 In particular we conclude estimates for a corresponding energy and estimates of dispersive type.
\end{abstract}
\thanks{The first author is supported by CNPq - Brazil with grant 246117/2012-5.}
\maketitle
\section{Introduction}

The study of wave models with lower order terms and their influence on energy and dispersive estimates for them 
has led to some interesting observations and also challenging problems. In this paper we will consider the Cauchy
problem
\begin{equation}
 u_{tt} - \Delta u + b(t)u_t + m(t) u
=0 ,\,\,\,u(0,x)=u_0(x),\,\,\,u_t(0,x)=u_1(x),
\end{equation}
for a damped Klein--Gordon equation with variable mass and dissipation and investigate the precise interplay between both coefficients and asymptotic properties of solutions as $t$ tends to infinity.   
For an overview on some results in this direction see \cite{Wir2010} and the monograph \cite{RW14}, we recall just some selected results to motivate the conditions imposed on the coefficient functions as well as the problems 
discussed later on.

One of the starting points for our considerations was the study of Reissig and Smith, \cite{RS05}, treating wave equations with bounded time-dependent speed of propagation and deriving $L^p$--$L^q$ decay estimates for their solutions. A second one is the treatment of wave equations with time-dependent dissipation of the second author
\cite{Wirth1} and \cite{Wirth2} introducing the classification of dissipation terms according to their strength and influence on the large-time behaviour of solutions. There the notion of {\em effective} and {\em non-effective} dissipation was used to distinguish between lower order terms giving asymptotically dominating or sub-ordinate contributions to the large-time behaviour. Let us assume for a moment that no mass term is present, $m=0$, and that $b$ is bounded, non-negative, sufficiently smooth and satisfies a condition of the form
\begin{equation}
    | \partial_t^k b(t) | \le C_k b(t) \left(\frac1{1+t}\right)^k,\qquad k=1,2.
\end{equation}
Then there are essentially two cases. Either
\begin{equation}\label{eq:1}
   \limsup_{t\to\infty} tb(t) < 1.
\end{equation}
Then solutions behave in an asymptotic sense like free waves multiplied by a decay factor, 
\begin{equation}
  \begin{pmatrix} \nabla u(t,x)  \\ u_t(t,x) \end{pmatrix} \sim \frac1{\lambda(t)}  \begin{pmatrix} \nabla v(t,x)  \\ v_t(t,x) \end{pmatrix},\qquad t\to\infty
\end{equation}
the asymptotic relation understood in an appropriate $L^p$-sense, and with $v$ a solution to the free wave equation $v_{tt}=\Delta v$ and $\lambda$
 given as
 \begin{equation}
    \lambda(t) = \exp\left(\frac12\int_0^t b(\tau)\d\tau\right).
 \end{equation}
 This is the above mentioned non-effective case. On the other hand, if
\begin{equation}
   \lim_{t\to\infty} tb(t) = \infty
\end{equation}
solutions to the damped wave equation are asymptotically related to solutions of a parabolic equation
\begin{equation}
  u(t,x) \sim w(t,x),\qquad t\to\infty,
\end{equation}
where $b(t)w_t = \Delta w$. This is made precise in the so-called diffusion phenomenon for damped waves, see 
\cite{Wirth3} or the work of Nishihara \cite{Nish} for the case of constant dissipation term. 

The main tool to prove such statements are asymptotic, but explicit, representations of solutions in terms of Fourier multipliers. In the present paper we will consider both mass and dissipation terms and concentrate on the non-effective case when solutions are still asymptotically hyperbolic and the derivation of $L^p$--$L^q$ decay estimates
depending on high-frequency asymptotics of solutions. 

A difference to earlier accounts is the more systematic study of the low-frequency parts based on asymptotic integration techniques. This allows to understand conditions like the constant appearing on the right-hand side of \eqref{eq:1}, or better, the construction of asymptotic low-frequency solutions in terms of leading order terms in\begin{equation}
   b(t) = \frac{b_0}{1+t} + o\left(\frac1{1+t}\right),\qquad    m(t) = \frac{m_0}{(1+t)^2} + o\left(\frac{1}{(1+t)^2}\right).
\end{equation}

The paper is organised as follows. In Section~\ref{sec2} we will give precise conditions on coefficients, outline our basic strategy and give asymptotic constructions of representation of solutions in different zones of the phase space.
In Section~\ref{sec3} we derive energy estimates, $L^p$--$L^q$ estimates and discuss their sharpness. Finally, the Appendix~\ref{secA} collects some useful asymptotic integration theorems for differential equations.

\section{Representations of solutions}\label{sec2}

We consider the following Cauchy problems for damped Klein--Gordon equations
\begin{equation}
 \label{KleinGordonDamping} u_{tt} - \Delta u + b(t)u_t + m(t) u
=0 ,\,\,\,u(0,x)=u_0(x),\,\,\,u_t(0,x)=u_1(x),
\end{equation}
where $(t,x) \in \R_+ \times \R^n$, $b=b(t)$ constitutes a dissipation term and $m=m(t)$ describes a mass term 
under the following basic assumptions. 
\begin{hypo} \label{Hypo6.1}
Suppose that $b,m \in C^{\ell} (\R_+)$ are real-valued and 
\begin{align}
\left| \partial_t^k b(t)  \right| \leq C_k \left( \frac{1}{1+t} \right)^{k+1},\qquad\text{and}\qquad 
\left| \partial_t^k m(t)   \right| \leq C_k \left( \frac{1}{1+t} \right)^{k+2}
\end{align}
holds true for all $k=0,1,\ldots,\ell$. The number $\ell$ will be specified later on. Some statements need a higher regularity.
\end{hypo}
\begin{hypo}\label{Hypo6.2}
Suppose that the following limits 
\begin{equation}
\lim_{t \rightarrow \infty} (1+t) b(t) = b_0 \mbox{\, \, and \, \,}  \lim_{t \rightarrow \infty} (1+t)^2 m(t) = m_0
\end{equation}
exist and that
\begin{equation}
  \int_1^\infty |t b(t) - b_0 |^\sigma \frac{ \d t}t  < \infty \mbox{\, \, and  \, \,  } \int_1^\infty {|t^2 m(t) - m_0 |^{\sigma}}\frac{\d t}t < \infty
\end{equation}
holds true with exponent $\sigma$ satisfying
\begin{equation*}
\text{\bf (A1)}\qquad \sigma=1 \qquad\qquad\text{or}\qquad\qquad \text{\bf (A2)} \qquad \sigma\in(1,2].
\end{equation*}
\end{hypo}

Results will depend on relations between the constants $b_0$ and $m_0$. It will not be necessary to restrict considerations to $b_0
\ge0$ and $m_0\ge0$, results will however depend on the constraint $4m_0>b_0(b_0-2)$ or additional conditions imposed on initial data. See, e.g., the statements of 
Theorems~\ref{Nestea6}, \ref{Nestea6a} or \ref{thm:33} for further details. 

We further define the auxiliary function 
\begin{equation}
\lambda(t) = \exp \left(\frac{1}{2} \int_0^t b(\tau) \d\tau \right)
\end{equation}
related to the dissipative term $b(t)u_t$. It will play an important role in the resulting estimates. Under part (A1) of Hypothesis 2 it follows that
\begin{equation}\label{eq:2.6}
   \lambda(t) \asymp (1+t)^{\frac{b_0}2},\qquad t\to\infty,
\end{equation}
where $f(t)\asymp g(t)$ as $t\to\infty$ means that the limit of the quotient $f(t)/g(t)$ exists. 
When assuming (A2) a further sub-polynomial correction term appears.

\subsection{Zones and general strategy}
Applying the partial Fourier transform in \eqref{KleinGordonDamping} with respect to the spatial variables, we obtain
the ordinary differential equation
\begin{equation}
 \label{KleinGordonDamping2} \widehat{u}_{tt} + |\xi|^2 \widehat{u} + b(t)\widehat{u}_t + m(t) \widehat{u}
=0 ,\,\,\,\widehat{u}(0,\xi)=\widehat{u}_0(\xi),\,\,\,\widehat{u}_t(0,\xi)=\widehat{u}_1(\xi)
\end{equation}
parameterised by the frequency variable $\xi$. We are going to construct asymptotically a parameter-dependent solution to this equation. In order to do so, we 
divide the extended phase space $[0,\infty)\times\mathbb{R}^n_\xi$ into three
zones depending on a constant $N>0$,
\begin{align}
{\mathcal Z}_{\rm diss}(N)
    &= \{(t,\xi) \in [0,\infty) \times \R^n  : (1+t)|\xi| \leq N\},\notag\\
{\mathcal Z}_{\rm hyp}^s(N)
    &= \{(t,\xi) \in [0,\infty) \times \R^n : |\xi| \leq N \leq (1+t)|\xi|\},\\
{\mathcal Z}_{\rm hyp}^\ell(N)
    &= \{(t,\xi) \in [0,\infty) \times \R^n : |\xi| \geq N\}.\notag
\end{align}
The constant $N$ will be specified later on.
In the zone ${\mathcal Z}_{\rm hyp}^\ell(N)$ we consider only large
frequencies and in the zones ${\mathcal Z}_{\rm diss}(N)$ and ${\mathcal Z}_{\rm hyp}^s(N)$ we consider
small frequencies.
Furthermore, the boundary between zones ${\mathcal Z}_{\rm diss}(N)$ and ${\mathcal Z}_{\rm hyp}^s(N)$ is given by the implicitly defined function
\begin{equation} 
\theta_{|\xi|} : (0,N]\to [0,\infty), \,\,\,\, (1+\theta_{|\xi|})|\xi|=N.
\end{equation}
We also put  $\theta_0=\infty$, and $\theta_{|\xi|}=0$ for any~$|\xi|\geq N$.
In order to localise the consideration to the three parts of the extended phase space, we further introduce a function $\chi \in C^{\infty} (\R_+)$
such that $\chi(t) = 1$ for $t \leq 1,$ $\chi(t) = 0$ for $t \geq 2$ and $\chi'(t) \leq 0$, and define the cut-off  functions $\varphi_{\rm diss}$, $\varphi_{\rm hyp}^\ell$ and $\varphi_{\rm hyp}^s$ of the zones ${\mathcal Z}_{\rm diss}(N)$, ${\mathcal Z}_{\rm hyp}^\ell(N)$ and
${\mathcal Z}_{\rm hyp}^s(N)$ by
\begin{align}
\varphi_{\rm diss}(t,\xi) &= \chi\left(|\xi|N^{-1}\right) \chi\left((1+t)|\xi|N^{-1}\right)\notag \\
\varphi_{\rm hyp}^s(t,\xi) &= \chi\left(|\xi|N^{-1}\right) \left(1-\chi\left((1+t)|\xi|N^{-1}\right)\right) \\
\varphi_{\rm hyp}^\ell(\xi) &= 1 - \chi\left(|\xi|N^{-1}\right)\notag
\end{align}
such that $\varphi_{\rm diss}(t,\xi) + \varphi_{\rm hyp}^\ell(\xi)+\varphi_{\rm hyp}^s(t,\xi) = 1.$
We consider the micro-energy 
\begin{equation} \label{mainmicro6}
U(t,\xi) = \left(  h(t,\xi) \widehat{u}, \D_t\widehat{u} \right)^T,
\end{equation}
where
\begin{equation}
 h(t,\xi) = \frac{N}{1+t} \varphi_{\rm diss}(t,\xi)+ |\xi| \left(\varphi_{\rm hyp}^\ell(\xi)+ \varphi_{\rm hyp}^s(t,\xi) \right)
\end{equation}
is a suitable time-dependent version of the usual Sobolev weight $(1+|\xi|^2)^{1/2}$ and $\D_t = -\i \partial_t$ denotes the Fourier derivative.

In the hyperbolic zone we apply a diagonalization procedure to a
first-order system corresponding to equation \eqref{KleinGordonDamping2} in order to derive a representation 
for the fundamental solution. We follow some ideas of Wirth \cite{Wirth1} and Yagdjian \cite{Y}. We will consider a system
with a coefficient matrix composed of a diagonal main part and a remainder part.
The goal of this diagonalization is to keep the diagonal part in every step of the
diagonalization and to improve the remainder terms.

To derive the asymptotic behavior of the fundamental solution 
 to \eqref{KleinGordonDamping2} in the dissipative zone we will perform, 
for $L^1$ condition (A1), one step of diagonalization and 
apply the Levinson Theorem \ref{thm:Lev1} and, for $L^{\sigma}$ condition (A2), we will apply the Hartman--Wintner Theorem~\ref{HartWintner}.
For the $L^{\sigma} $ condition we need one more step of diagonalization (see proof of  Theorem~\ref{HartWintner}).

\subsection{Treatment in the dissipative zone}

In the dissipative zone the micro-energy \eqref{mainmicro6} becomes
\begin{align}
U(t,\xi) = \left(  \frac{N}{1+t} \widehat{u}, \D_t \widehat{u}\right).
\end{align}
Therefore, equation \eqref{KleinGordonDamping2} rewrites as system
\begin{align} \label{PDKGD}
\D_t U(t,\xi) = \widetilde{A}(t,\xi) U(t,\xi) =\begin{pmatrix}
          \frac{\i}{1+t}  & \frac{N}{1+t} \\
         \frac{   (1+t)\left( |\xi|^2 + m(t) \right) }N &  \i b(t)  \end{pmatrix}U(t,\xi)
\end{align}
with coefficient matrix $\widetilde A(t,\xi)$ within ${\mathcal Z}_{\rm diss}(N)$. In order to estimates its fundamental 
solution $\mathcal{E}(t,s,\xi)$ we apply Levinson's Theorem~\ref{thm:Lev1}. Note, that the matrix $\widetilde A(t,\xi)$ behaves like $(1+t)^{-1}$
in the zone and therefore it reasonable to rewrite it as Fuchs type system
\begin{align}\label{Black6}
(1+t)\partial_t U(t,\xi) = \left( A + R(t,\xi) \right) U(t,\xi)
\end{align}
with matrices
\begin{equation}
  A = \begin{pmatrix}
     -1 & \i N  \\    \frac{   \i m_0}N  &  - b_0 
  \end{pmatrix}
\end{equation}
and 
\begin{equation}
   R(t,\xi) = \begin{pmatrix}
          0 & 0 \\
           \frac{\i(1+t)^2 |\xi|^2 +\i \left( (1+t)^2 m(t) - m_0 \right)}N &    b_0- (1+t)b(t) \end{pmatrix}.
\end{equation}
By Hypothesis  \ref{Hypo6.2}  in the form (A1) and the definition of the zone we know that
\begin{equation}
   \sup_{|\xi|<N}   \int_1^{\theta_{|\xi|}} \| R(t,\xi)\| \frac{\d t}{t} < \infty
\end{equation}
and $R(t,\xi)$ is a remainder term in the sense of Theorem~\ref{thm:Lev1}. Furthermore, as $\mathrm{tr}\, A=-1-b_0$ and $\det A= b_0 + m_0$ the eigenvalues of $A$ are given as
\begin{align}\label{eq:mu}
   \mu_\pm = -\frac{b_0+1}2 \pm \sqrt{\frac{(b_0-1)^2}4 - m_0}.
\end{align}
In particular we see that  
\begin{equation}\label{eq:cond-diffmu}
4m_0 \ne (b_0-1)^2
\end{equation}
implies that the eigenvalues are distinct.
\begin{thm} \label{WanWir}
Assume Hypothesis \ref{Hypo6.2} with $\sigma=1$ together with \eqref{eq:cond-diffmu}. Then the matrix-valued fundamental solution of
the system \eqref{Black6} satisfies
\begin{align}
\label{AsymptBehaPD1}
\|\mathcal{E}(t,s,\xi)\| \lesssim \left(\frac{1+t}{1+s}\right)^{\Re\mu_+}
\end{align}
uniformly  in $0\le s\le t$ and $(t,\xi)\in {\mathcal Z}_{\rm diss}(N)$.
\end{thm}
\begin{proof}
This follows from Theorem~\ref{thm:Lev1} applied to \eqref{Black6} with $R(t,\xi)$ extended by zero outside ${\mathcal Z}_{\rm diss}(N)$. 
 To simplify notation, we denote by $e_\pm$ the two normalised eigenvectors corresponding to $\mu_\pm$. 
 From $\mu_+\ne\mu_-$ we conclude that there exist two linearly independent solutions to \eqref{Black6} of the form
\begin{equation}
    U_\pm(t,\xi) = (e_\pm + o(1)) (1+t)^{\mu_\pm},\qquad t\to\infty
\end{equation}
within ${\mathcal Z}_{\rm diss}(N)$ and uniformly in $\xi$.
Constructing the fundamental matrix as in Remark~\ref{ConsFundSol}, we see that 
\begin{equation}
    \mathcal E(t,0,\xi) = \big( U_-(t,\xi) | U_+(t,\xi) \big) \big( U_-(0,\xi) | U_+(0,\xi)\big)^{-1}, 
\end{equation}
and hence, we obain
\begin{equation}
    \| \mathcal E(t,0,\xi) \| \lesssim   (1+t)^{\Re\mu_+}
\end{equation}
for any  $(t,\xi)\in {\mathcal Z}_{\rm diss}(N)$. Using the scaling from Remark~\ref{rem:scaling}
(taking into account the shift in time) we obtain \eqref{AsymptBehaPD1} uniformly in $0\le s\le t\le \theta_{|\xi|}$.
\end{proof}

In order to treat the form (A2) of Hypothesis 2 by the Hartmann--Wintner Theorem~\ref{HartWintner}, we need to ensure that $\Re\mu_+\ne\Re\mu_-$. This happens if both are real and distinct. The latter is equivalent to
\begin{equation}\label{eq:cond-realmu}
  4m_0 < (b_0-1)^2.
\end{equation} 

\begin{thm} \label{WanWir2}
Assume Hypothesis \ref{Hypo6.2} with $\sigma\in(1,2]$ together with \eqref{eq:cond-realmu}. Let further $\sigma'$ be the dual Lebesgue index
to $\sigma$. Then the fundamental solution of
the system \eqref{Black6} satisfies
\begin{align}
\label{AsymptBehaPDp}
\|\mathcal{E}(t,s,\xi)\| \lesssim \left(\frac{1+t}{1+s}\right)^{\mu_+} \exp\left(C \left( \ln \frac{1+t}{1+s}  \right)^{\frac{1}{\sigma'}}\right)
\end{align}
uniformly in $0\le s\le t$ and $(t,\xi)\in {\mathcal Z}_{\rm diss}(N)$.
\end{thm}

\begin{proof}
As in the previous case we extend $R(t,\xi)$ by zero outside ${\mathcal Z}_{\rm diss}(N)$ and denote by $e_\pm$ normalised eigenvectors of $A$
corresponding to $\mu_\pm$. Forming the unitary matrix $P=(e_-|e_+)$ with these eigenvectors as columns
and defining $\widetilde R(t,\xi) = P^{-1} R(t,\xi) P$ allows to rewrite \eqref{Black6} in the new unknown $\widetilde U(t,\xi) = P U(t,\xi)$ as
\begin{equation}
    (1+t)\partial_t  \widetilde U(t,\xi) = \big( \diag(\mu_-,\mu_+) + \widetilde R(t,\xi) \big)\widetilde U(t,\xi)
\end{equation}
We apply Theorem~\ref{HartWintner} to this system. As $\mu_\pm$ are real and distinct, they clearly satisfy \eqref{eq:LevC2}. Furthermore, 
the matrix $\widetilde R(t,\xi)$ contains combinations of $(1+t)b(t)-b_0$ and $(1+t)^2m(t)-m_0$ controlled by (A2) and terms of the form  $(1+t)^2|\xi|^2$ which are uniformly bounded and integrable with respect to $\d t/t$ by the definition of the zone.
Hence, Hypothesis \ref{Hypo6.2} in the form (A2) implies \eqref{eq:HWcond} with $\sigma\in(1,2]$. Therefore, Theorem~\ref{HartWintner} applies and gives a
a matrix $N(t,\xi) \in L^\sigma(\R_+,\d t/t)$ transforming \eqref{Black6} for $t\ge t_0$ into Levinson form
\begin{equation}
    (1+t) \partial_t V(t,\xi)  = \big( \diag(\mu_-+\widetilde r_{--} ,\mu_++\widetilde r_{++}) + \widetilde R_1(t,\xi) \big)V(t,\xi) 
\end{equation}
in the new unknown $V(t,\xi) = (I+N(t,\xi))^{-1} \widetilde U(t,\xi)$ and with the new remainder $\widetilde R_1\in L^1([t_0,\infty),\d t/t)$. By $\widetilde r_{--}(t,\xi)$
and $\widetilde r_{++}(t,\xi)$ we denote the diagonal entries of $\widetilde R(t,\xi)$. The new diagonal part satisfies the dichotomy condition
\eqref{eq:LevCond1}, the additional diagonal entries satisfy by H\"older's inequality
\begin{equation}
    \int_s^t |\widetilde r_{++}(\tau,\xi) | \frac{\d\tau}{1+\tau}  \le C \left(\ln \frac{1+t}{1+s}\right)^{\frac1{\sigma'}} 
\end{equation}
with $\sigma'$ the dual index and are thus small compared to 
\begin{equation}
    \int_s^t (\mu_+-\mu_-)\frac{\d\tau}{1+\tau} = (\mu_+-\mu_-)  \left(\ln \frac{1+t}{1+s}\right).
\end{equation}
Hence, Levinson's theorem~\ref{thm:Lev1} yields a fundamental system of solutions together with the estimate
\begin{equation}
   \| \mathcal E_V(t,t_0,\xi) \| \le (1+t)^{\mu_+} \exp\left(  C \left(\ln(1+t)\right)^{\frac1{\sigma'}}  \right), \quad t\ge t_0,
\end{equation}
for the matrix-valued fundamental solution to the transformed system. The scaling argument from Remark~\ref{rem:scaling}
extends this estimate to variable starting times $t_0\le s\le t\le \theta_{|\xi|}$ as 
\begin{equation}
   \| \mathcal E_V(t,s,\xi) \| \lesssim \left(\frac{1+t}{1+s}\right)^{\mu_+} \exp\left(  C \left(\ln \frac{1+t}{1+s}\right)^{\frac1{\sigma'}}  \right).
\end{equation}
Transforming back to the original system combined with compactness of the remaining bit of ${\mathcal Z}_{\rm diss}(N)$ where the transform was not defined yields the desired statement. The theorem is proved.
\end{proof}

\begin{rem} \label{rota98}
If $2\Re \mu_+<-{b_0}$, i.e., if
\begin{equation}\label{eq:cond-smallmu}
    b_0(b_0-2) < 4m_0, 
\end{equation}
then Theorems~\ref{WanWir} and \ref{WanWir2} imply
\begin{align}\label{eq:Eest-lambda-Zpd}
\|  \mathcal{E}(t,s,\xi) \| \lesssim \frac{\lambda(s)}{\lambda(t)}
\end{align}
for all $0\le s\le t$ and  $(t,\xi) \in {\mathcal Z}_{\rm diss}(N).$ In the first case this is obvious, while in the second case we observe that 
 for all $\varepsilon > 0$ there exists a constant $c_{\varepsilon}$ such that
\begin{align}
\exp\left(C \left( \ln \frac{1+t}{1+s}   \right)^{\frac{1}{\sigma'}}\right) \leq c_{\varepsilon} \left( \frac{1+t}{1+s} \right)^{\varepsilon}.
\end{align} 
Therefore
\begin{multline}
\left( \frac{1+t}{1+s}\right)^{\mu_+}  \exp\left(C \left( \ln \frac{1+t}{1+s}   \right)^{\frac{1}{\sigma'}}\right)
\lesssim \left( \frac{1+t}{1+s}
\right)^{\mu_+ + \varepsilon } \\
 \lesssim   \left( \frac{1+t}{1+s} \right)^{-\frac{b_0}{2} - \varepsilon }  \lesssim \frac{\lambda(s)}{\lambda(t)}
\end{multline}
uniformly in $0\le s\le t$.
\end{rem}

\subsection{The zone-boundary} In order to combine the estimates from the dissipative zone with the treatment in the hyperbolic 
zone, we need one further estimate. It is conditional in the sense that it is entirely based on the final estimate from the dissipative zone and not on the precise assumptions used to prove it. It is also the first statement using Hypothesis  \ref{Hypo6.1}.

\begin{lem} \label{PeqLpLqIm}
Assume  Hypothesis \ref{Hypo6.1} and Hypothesis \ref{Hypo6.2} in combination with \eqref{eq:cond-smallmu}. Then for $|\xi| \leq N$ the symbol-like estimates
\begin{align}
\big\| \D_\xi^\alpha    \mathcal{E}(\theta_{|\xi|},0,\xi)   \big\| \leq C_\alpha \frac{1}{\lambda(\theta_{|\xi|})} |\xi|^{-|\alpha|}
\end{align}
are valid for all $|\alpha| \leq \ell$.
\end{lem}
\begin{proof}
To prove this fact we use Duhamel's formula for $\xi-$derivatives of \eqref{PDKGD}. Let
first $|\alpha|=1.$ Then $\D_t\D_\xi \mathcal E = \left( \D_\xi^\alpha \widetilde{A} \right) \mathcal{E} + \widetilde{A}  \left( \D_\xi^\alpha \mathcal{E}\right)$
and thus using $\D_\xi^\alpha \mathcal{E}(0,0,\xi)=0$ we obtain the representation
\begin{align}
\D_\xi^\alpha \mathcal{E}(t,0,\xi) =\i \int_0^t  \mathcal{E}(t,\tau,\xi) \left( \D_\xi^\alpha \widetilde{A}(\tau,\xi) \right)\mathcal{E}(\tau,0,\xi) \d\tau.
\end{align}
Since $\| \D_\xi^\alpha \widetilde{A}(t,\xi)  \| \lesssim 1$ this implies from \eqref{eq:Eest-lambda-Zpd} the estimate
\begin{align}
\| \D_\xi^\alpha \mathcal{E}(t,0,\xi)  \| \lesssim \frac{t}{\lambda(t)} \lesssim  \frac{1}{\lambda(t)}|\xi|^{-1}
\end{align}
uniformly on ${\mathcal Z}_{\rm diss}(N)$.

For $|\alpha| = \ell > 1$ we use Leibniz formula to represent $\D_\xi^\alpha \mathcal{E}(t,0,\xi) $ by a corresponding
Duhamel integral using a sum of terms $\D_\xi^{\alpha_1} \widetilde{A}(t,\xi) \D_\xi^{\alpha_2} \mathcal{E}(t,0,\xi) $
for $|\alpha_1|+|\alpha_1| \leq \ell$, $\alpha_2 < \alpha$ and apply induction over $\ell$ to obtain
\begin{align}
\| \D_\xi^\alpha \mathcal{E}(t,0,\xi)    \| \lesssim \frac{1}{\lambda(t)} |\xi|^{-|\alpha|}
\end{align}
uniformly within ${\mathcal Z}_{\rm diss}(N)$.

Derivates of $\mathcal{E}$ with respect to $t$ are estimated directly by the differential equation
and $\| \widetilde A(t,\xi)\| \lesssim (1+t)^{-1}$ giving
\begin{equation}
  \| \partial_t^k \D_\xi^\alpha \mathcal E(t,0,\xi) \| \le \frac1{\lambda(t)}  \left(\frac1{1+t}\right)^k|\xi|^{-|\alpha|}.
\end{equation} 
But now the statement follows from Fa\`{a} di Bruno's formula combined with 
\begin{align}\label{eq:theta-der}
|\D_\xi^\alpha \theta_{|\xi|}| \lesssim |\xi|^{-1-|\alpha|}.
\end{align}
\end{proof}

\begin{rem}
\label{defSymbClasMT}
The result of the Lemma \ref{PeqLpLqIm} can be reformulated in the following form.
The symbol $\lambda(\theta_{|\xi|}) \mathcal{E}(\theta_{|\xi|},0,\xi) $ is an element
of the homogeneous symbol class
\begin{align}
\dot{S}_{\ell}^0 = \{ m \in C^\infty(\R^n \setminus \{0\})\;:\; |\D_\xi^\alpha m(\xi) | \leq C_\alpha |\xi|^{-|\alpha|} \;\text{for all $ |\alpha| \leq \ell$}    \}
\end{align}
of order zero and restricted smoothness $\ell $.
\end{rem}

\subsection{Treatment in the hyperbolic zone}
First, we recall the definition of the hyperbolic symbol class $\mathcal S_N^{\ell}\{m_1,m_2\}$ from \cite{RW11} and \cite{RW14}.

\begin{defn}
The time-dependent amplitude function $a=a(t,\xi)$ belongs to the hyperbolic symbol class $\mathcal S_N^{\ell}\{m_1,m_2\}$ with restricted smoothness $\ell$
if it satisfies the symbol estimates
\begin{equation}
\label{mfourier}\left|\D_t^k\D_\xi^\alpha a(t,\xi)\right| \leq C_{k,\alpha}|\xi|^{m_1-|\alpha|} \left( \frac{1}{1+t}\right)^{m_2+k}
\end{equation}
for all $(t,\xi) \in {\mathcal Z}_{\rm hyp}(N)$, all non-negative integers $k \leq \ell$ and all multi-indices $\alpha \in \mathbb{N}^n$.
We will further use the notation 
\begin{equation} 
     \mathcal H^{\ell}_N \{k\} = \bigcap_{m_1+m_2=k} \mathcal S^{\ell}_N\{m_1,m_2\}
\end{equation}   
and 
$\mathcal S_N\{m_1,m_2\}$ as short-hand for $\mathcal S_N^{\infty}\{m_1,m_2\}$ and similarly for $\mathcal H_N\{k\}$. 
\end{defn}

\subsubsection{Diagonalisation}
Within the hyperbolic zone the micro-energy \eqref{mainmicro6} satisfies the system
\begin{equation} \label{Stutt}
\D_t U = A(t,\xi) U
\end{equation}
with
\begin{equation}
 A(t,\xi) =\begin{pmatrix}
          0  & |\xi| \\
           |\xi| + \frac{m(t)}{|\xi|}  &  \i b(t)  \end{pmatrix} \mod \mathcal H_N\{1\}
\end{equation}
as consequence of $h(t,\xi) = |\xi| \mod \mathcal H_N\{1\}$. Note, that $A(t,\xi)\in \mathcal S_N^{\ell}\{1,0\}$.
We denote by $\mathcal{E}(t,s,\xi)$ the fundamental solution to \eqref{Stutt}, i.e., the matrix-valued solution to
\begin{align} \label{Stuttg}
\D_t \mathcal{E}(t,s,\xi)  = A(t,\xi)   \mathcal{E}(t,s,\xi), \qquad
                    \mathcal{E}(s,s,\xi) = \mathrm I \in \C^{2\times 2}
\end{align}
for $t\ge s$ and $(s,\xi)\in {\mathcal Z}_{\rm hyp}(N)$. In order to estimate $\mathcal E(t,s,\xi)$ we follow \cite{RW11} and  apply a diagonalisation scheme within the hyperbolic symbol classes $\mathcal S_N^{\ell}\{\cdot,\cdot\}$.  We sketch the main steps, but refer for further details to \cite{RW14} and \cite{RW11}.

{\sl Preliminary step.} In the first step we use the matrix
\begin{align}
M = \frac{1}{\sqrt{2}}\begin{pmatrix}
          1  & -1 \\
           1   &  1  \end{pmatrix}, \qquad M^{-1} = \frac{1}{\sqrt{2}} \begin{pmatrix}
          1  & 1 \\
           -1   &  1  \end{pmatrix},
\end{align}
to transform the principal part of  the system. This yields for $V(t,\xi) = M^{-1}U(t,\xi)$ the new system
\begin{align}\label{firstStep}
\D_tV(t,\xi) = \big(D(\xi) +  B(t) + C(t,\xi) \big)V(t,\xi)
\end{align}  
with coefficient matrices
\begin{align}
D(\xi) &= \diag(|\xi|, -|\xi|)\in \mathcal S_N\{1,0\},\\
 B(t) &= \frac{\i b(t)}{2}\begin{pmatrix}
          1  &  1 \\
           1   &  1  \end{pmatrix} \in \mathcal S_N^{\ell}\{0,1\}\\
\intertext{ and}
  C(t,\xi) &= \frac{m(t)}{2|\xi|}\begin{pmatrix}
          1  & -1 \\
           1   &  -1  \end{pmatrix}\in \mathcal S_N^{\ell}\{-1,2\}.
\end{align}
The matrix $D(\xi)$ gives the diagonal principal part and our aim is to apply further transformations to improve the symbolic behaviour of the remainders.  

{\sl First step.} We give the first step in detail. We want to find a matrix-valued symbol $N^{(1)}(t,\xi)$ from $\mathcal S_N^{\ell}\{-1,1\}$ and a diagonal matrix
$F^{(0)}(t,\xi)$ from $\mathcal S_N^{\ell}\{0,1\}$ such that the operator identity
\begin{multline}
   B^{(1)}(t,\xi) =   \big( \D_t - D(\xi) -  B(t) - C(t,\xi) \big) (\mathrm I+N^{(1)}(t,\xi)) \\
   - (\mathrm I+N^{(1)}(t,\xi)) \big( \D_t - D(\xi) - F^{(0)}(t,\xi)\big)
\end{multline}
defines a matrix $B^{(1)}(t,\xi)$ from $\mathcal S_N^{\ell-1}\{-1,2\}$. In order for this to happen, we have to satisfy the commutator identity
\begin{equation}
   [D(\xi), N^{(1)}(t,\xi)] + B(t) = F^{(0)}(t,\xi). 
\end{equation}
As $D(\xi)$ is diagonal, the commutator on the left has zero diagonal entries. Therefore, $F^{(0)}(t,\xi)=\diag B(t)$ (where $\diag$ denotes the diagonal part of the matrix). In particular it follows that $F^{(0)}(t,\xi)$ is independent of $\xi$ and is of the desired class. Similarly, the off-diagonal parts of $N^{(1)}(t,\xi)$ are
determined by the commutator equation while we are free to choose the diagonal entries. Requiring in addition that $\diag N^{(1)}(t,\xi)=0$, we obtain
\begin{equation}
    N^{(1)}(t,\xi) = \frac{\i b(t)}{4|\xi|} \begin{pmatrix} 0& 1 \\ -1 & 0 \end{pmatrix}
\end{equation}
and it is clear that this matrix belongs to the symbol class $\mathcal S_N^{\ell}\{-1,1\}$.

{\sl Iterative improvements.}
We construct recursively matrices $N^{(k)} (t,\xi)$ from the classes $\mathcal S^{\ell+1-k}\{-k,k\}$ and diagonal matrices $F^{(k)}(t,\xi)$ from $\mathcal S_N^{\ell+1-k}\{1-k,k\}$
such that for 
\begin{equation}
  N_k(t,\xi) = \mathrm I+\sum_{j=1}^k N^{(j)}(t,\xi), \qquad   F_{k-1}(t,\xi) = \sum_{j=0}^{k-1} F^{(j)}(t,\xi),   
\end{equation}
the operator identity
\begin{multline}
   B^{(k)}(t,\xi) =   \big( \D_t - D(\xi) -  B(t) - C(t,\xi) \big)N_k(t,\xi)\\
   - N_k(t,\xi) \big( \D_t - D(\xi) - F_{k-1}(t,\xi)\big)
\end{multline}
yields a remainder $B^{(k)}(t,\xi) \in \mathcal S_N^{\ell-k}\{-k,k+1\}$. This again implies a commutator identity and from that 
 \begin{align}
    \label{nayla}  F^{(k-1)}(t,\xi)  &= \diag B^{(k-1)}(t,\xi), \\
      N^{(k)}(t,\xi) &=\frac1{2|\xi|} \begin{pmatrix}
                   0 & -(B^{(k-1)}(t,\xi))_{12} \\
                     {(B^{(k-1)}(t,\xi))_{21}} & 0 \end{pmatrix},
 \end{align}
 where $(\cdot)_{ij}$ stands for the $ij$-entry of the matrix.  By induction we also obtain that all matrices belong to the desired symbol classes.
\begin{prop} \label{judô2}
Assume  Hypothesis \ref{Hypo6.1} with derivatives up to order $\ell$. 
Then $N^{(k)} \in \mathcal S_N^{\ell-k+1}\{-k,k\}$ and $B^{(k)}, F^{(k)} \in \mathcal S_N^{\ell-k}\{-k,k+1\}$ holds true for all $k=1,\ldots,\ell$.
Moreover, for any such $k$ we find a zone constant $N$ such that the matrix $N_k(t,\xi)$ is invertible in
${\mathcal Z}_{\rm hyp}(N)$ and $N_k(t,\xi)^{-1}\in \mathcal S_N^{\ell-k+1}\{0,0\}$.
\end{prop}

\subsubsection{Fundamental solutions}
In the following we choose $N$ large enough for $N_k$ to be invertible within the zone ${\mathcal Z}_{\rm hyp}(N)$. 
If we denote $R_k(t,\xi)= - N_k(t,\xi)^{-1}B^{(k)}(t,\xi)$ and consider the new unknown $V_k(t,\xi) = N_k(t,\xi)^{-1} V(t,\xi)$ then
we obtain the transformed system
\begin{equation}\label{eq:2.59}
  \D_t V_k(t,\xi) = \big( D(\xi) + F_{k-1}(t,\xi) + R_k(t,\xi) \big) V_k(t,\xi)
\end{equation}
within ${\mathcal Z}_{\rm hyp}(N)$ with diagonal $F_{k-1}(t,\xi)\in \mathcal S_N^{\ell+1-k}\{0,1\}$ and with non-diagonal remainder $R_k(t,\xi)\in \mathcal S_N^{\ell-k}\{-k,k+1\}$. It remains to estimate its fundamental solution $\mathcal E_k(t,s,\xi)$. This follows along the lines of \cite{RW11}.

\begin{thm} \label{Lema6}
Assume Hypothesis \ref{Hypo6.1}. Then the fundamental solution $\mathcal{E}_k(t,s,\xi)$, $k\ge 1$, of the diagonalized system \eqref{eq:2.59}
can be represented as
\begin{align}\label{wendel}
\mathcal{E}_k(t,s,\xi) = \frac{\lambda(s)}{\lambda(t)} \mathcal{E}_0(t,s,\xi) \mathcal{Q}_k(t,s,\xi),
\end{align}
for $t\geq s$ and $(s,\xi) \in {\mathcal Z}_{\rm hyp}(N),$ where
\begin{enumerate}
 \item the function 
 \begin{equation}
 \lambda(t)= \exp \left( \frac{1}{2} \int_0^t b(\tau)\, \d\tau  \right)
 \end{equation}
  describes the main influence of the dissipation $b$,
  \item the matrix $\mathcal{E}_0(t,s,\xi)$  given by
  \begin{align} 
  \label{neve}
                        \mathcal{E}_0(t,s,\xi) =\begin{pmatrix}
                   \e^{\i(t-s)|\xi|} & 0\\
                     0 & \e^{-\i(t-s)|\xi|} \end{pmatrix}
  \end{align}
  is the fundamental solution of the free wave equation,
 \item the matrix $\mathcal{Q}_k(t,s,\xi)$  satisfies  for all multi-indices $|\alpha|\leq \min\{ k-1, \ell-k-1\}$ the symbol like estimates
    \begin{equation} \label{xiDerivQ}
      \left\|  \D_\xi^\alpha \mathcal{Q}_k(t,s,\xi)   \right\| \leq C_{\alpha} |\xi|^{-|\alpha|}
    \end{equation}
uniformly in $t\ge s\ge \theta_{|\xi|}$ and
    \begin{equation} \label{xiDerivQ2}
      \left\|  \D_\xi^\alpha \mathcal{Q}_k(t,\theta_{|\xi|},\xi)   \right\| \leq C_{\alpha} |\xi|^{-|\alpha|}
    \end{equation}
uniformly in $(t,\xi)\in {\mathcal Z}_{\rm hyp}(N)$.
\end{enumerate}
Furthermore, $\mathcal Q_k(t,s,\xi)$ is invertible and converges for $t\to\infty$ to the invertible matrix $\mathcal Q_k(\infty,s,\xi)$ locally uniform with respect to $\xi\ne0$.
\end{thm}
\begin{proof}
The fundamental solution to the main diagonal part $\D_t - D(\xi)-F_0(t)$ is given by 
\begin{equation}
    \frac{\lambda(s)}{\lambda(t)} \mathcal E_0(t,s,\xi),
\end{equation}
therefore the matrix $\mathcal Q_k(t,s,\xi)$ in \eqref{wendel} solves 
\begin{equation}
       \D_t \mathcal Q_k(t,s,\xi) =\mathcal R_k(t,s,\xi)  \mathcal Q_k(t,s,\xi),\qquad \mathcal Q_k(s,s,\xi)=\mathrm I,
\end{equation}
for $\theta_{|\xi|}\le s, t$ and with
\begin{equation}
    \mathcal R_k(t,s,\xi) = \mathcal E_0(s,t,\xi) \big( F_{k-1}(t,\xi) - F_0(t,\xi) + R_k(t,\xi) \big) \mathcal E_0(t,s,\xi).
\end{equation}
{\sl Uniform bounds.} As symbols from $\mathcal S_N^{0}\{-1,2\}$ are uniformly integrable in $t$ within the hyperbolic zone, it follows that
\begin{equation}
\sup_{(s,\xi)\in {\mathcal Z}_{\rm hyp}(N)}  \int_{s}^\infty \|\mathcal R_k(t,s,\xi) \| \d t = C < \infty
\end{equation}
and $\mathcal Q_k(t,s,\xi)$ is uniformly bounded in $t, s\ge \theta_{|\xi|}$ as consequence of the representation by Peano--Baker formula
\begin{equation}\label{frisch}
  \mathcal Q_k(t,s,\xi) = \mathrm I + \sum_{j=1}^\infty \i^j \int_s^t \mathcal R_k(t_1,s,\xi)\int_{s}^{t_1} \cdots \int_s^{t_{j-1}} \mathcal R_k(t_j,s,\xi)\, \d t_j \cdots \d t_2 \d t_1. 
\end{equation}
Furthermore,  Liouville theorem yields
\begin{equation}
 |  \det \mathcal Q_k(t,s,\xi)| = \left|\exp\left( \i \int_s^t \mathrm{tr}\, \mathcal R_k(\tau,s,\xi) \,\d\tau \right) \right| \ge \exp(-2C) > 0
\end{equation}
and thus uniform invertibility of $\mathcal Q_k(t,s,\xi)$. The convergence for $t\to\infty$ follows from the Cauchy criterion combined with \eqref{frisch}.

{\sl Estimates for derivatives.} For estimating derivatives, we have to treat the diagonal terms and the remainder separately. Writing
\begin{equation}\label{eq:2.70}
   \mathcal Q_k(t,s,\xi) = \exp\left(\i \int_s^t (F_{k-1}(\tau,\xi)- F_0(\tau,\xi)) \,\d\tau\right) \widetilde{\mathcal Q}_k(t,s,\xi),
\end{equation}
we obtain an equation 
\begin{equation}
       \D_t \widetilde{\mathcal Q}_k(t,s,\xi) =\widetilde{\mathcal R}_k(t,s,\xi)  \widetilde{\mathcal Q}_k(t,s,\xi),\qquad \widetilde{\mathcal Q}_k(s,s,\xi)=\mathrm I,
\end{equation}
with improved integrability properties of $\widetilde{\mathcal R}_k(t,s,\xi)$. The exponential in equation \eqref{eq:2.70} is uniformly bounded and behaves 
both in $(t,\xi)$ as well as in $(s,\xi)$ as symbol from $\mathcal S_N^{\ell-k+2}\{0,0\}$ uniformly in the remaining variable. This follows from Proposition~\ref{judô2}.

The matrix $\widetilde{\mathcal R}_k(t,s,\xi)$ has symbol-like estimates for large enough $k$, where the good behaviour of the remainder 
$R_k\in \mathcal S_N^{\ell-k}\{-k,k+1\}$ allows to compensate the badly behaving derivatives of $\mathcal E_0$. This implies
(in combination with the definition of the zone)
\begin{equation}
    \| \partial_t^{\beta_1} \partial_s^{\beta_2} \D_\xi^\alpha \widetilde{\mathcal R}_k(t,s,\xi)\| \le \left(\frac1{1+t}\right)^2 |\xi|^{-1-|\alpha|+|\beta|}
\end{equation}
for $|\alpha|-|\beta|\le k-1$ and $|\beta|\le  \ell-k-1$. Combined with Peano--Baker formula and estimate \eqref{eq:theta-der} for derivatives of the zone-boundary the estimates 
  \begin{equation} 
      \left\|  \D_\xi^\alpha \widetilde{\mathcal{Q}}_k(t,s,\xi)   \right\| \leq C_{\alpha} |\xi|^{-|\alpha|},\qquad |\alpha|\le k-1,
    \end{equation}
and
    \begin{equation} 
      \left\|  \D_\xi^\alpha \widetilde{\mathcal{Q}}_k(t,\theta_{|\xi|},\xi)   \right\| \leq C_{\alpha} |\xi|^{-|\alpha|},\qquad |\alpha|\le \min\{k-1,\ell-k-1\}
    \end{equation}
follow.  They imply the desired statements.
\end{proof}
\subsubsection{Transforming back to the original problem}
After constructing the fundamental solution $\mathcal{E}_k(t,s,\xi),$ we transform back to
the original problem and get in the hyperbolic zone the representation
\begin{align} \label{Repres6}
\mathcal{E}(t,s,\xi) = \frac{\lambda(s)}{\lambda(t)} M N_k(t,\xi)  \mathcal{E}_0(t,s,\xi) \mathcal{Q}_k(t,s,\xi) N_k(s,\xi)^{-1}M^{-1},
\end{align}
with uniformly bounded symbols $N_k, N_k^{-1} \in \mathcal S_N^{\ell -k -1}\{0,0\}$, unitary $\mathcal E_0$ and $\mathcal Q_k$ of known properties. For large frequencies this is used with $s=0$, while for $|\xi|\le N$ we have to take into account the fundamental solution constructed in the dissipative zone. This leads to
\begin{multline} \label{Repres6s}
\mathcal{E}(t,0,\xi) = \frac{1}{\lambda(t)} M N_k(t,\xi)  \mathcal{E}_0(t,\theta_{|\xi|},\xi)\\
\times  \mathcal{Q}_k(t,\theta_{|\xi|},\xi) N_k(\theta_{|\xi|},\xi)^{-1} M^{-1} \lambda(\theta_{|\xi|})\mathcal{E}(\theta_{|\xi|},0,\xi),
\end{multline}
for $t \ge  \theta_{|\xi|}$.

\subsection{Collecting the estimates}

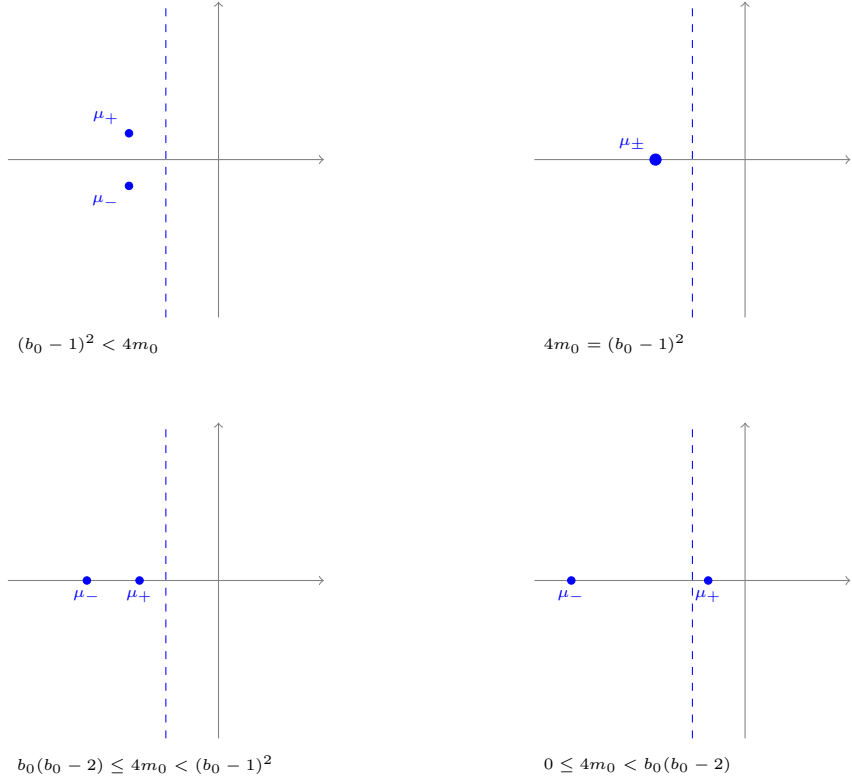
\begin{figure}
\begin{center}
\begin{tikzpicture}[scale=0.7]
\begin{scope}[xshift=10cm]
\draw[gray,->] (0,-3)--(0,3);
\draw[gray,->] (-4,0)--(2,0);
\draw[blue,dashed] (-1,-3)--(-1,3);
\filldraw[blue] (-0.7,0) circle(2pt) node [anchor = north] {\tiny $\mu_+$};
\filldraw[blue] (-3.3,0) circle(2pt) node [anchor = north] {\tiny $\mu_-$};
\draw[black] (-4,-3.5) node [anchor=west] {\tiny $0\le 4m_0<b_0(b_0-2)$};
\end{scope}
\begin{scope}
\draw[gray,->] (0,-3)--(0,3);
\draw[gray,->] (-4,0)--(2,0);
\draw[blue,dashed] (-1,-3)--(-1,3);
\filldraw[blue] (-1.5,0) circle(2pt) node [anchor = north ] {\tiny $\mu_+$};
\filldraw[blue] (-2.5,0) circle(2pt) node [anchor = north ] {\tiny $\mu_-$};
\draw[black] (-4,-3.5) node [anchor=west] {\tiny $b_0(b_0-2)\le 4m_0<(b_0-1)^2$};
\end{scope}
\begin{scope}[yshift=8cm]
\draw[gray,->] (0,-3)--(0,3);
\draw[gray,->] (-4,0)--(2,0);
\draw[blue,dashed] (-1,-3)--(-1,3);
\filldraw[blue] (-1.7,-0.5) circle(2pt) node [anchor = north east] {\tiny $\mu_-$};
\filldraw[blue] (-1.7,0.5) circle(2pt) node [anchor = south east] {\tiny $\mu_+$};
\draw[black] (-4,-3.5) node [anchor=west] {\tiny $(b_0-1)^2<4m_0$};
\end{scope}
\begin{scope}[xshift=10cm,yshift=8cm]
\draw[gray,->] (0,-3)--(0,3);
\draw[gray,->] (-4,0)--(2,0);
\draw[blue,dashed] (-1,-3)--(-1,3);
\filldraw[blue] (-1.7,0) circle(3pt) node [anchor = south east] {\tiny $\mu_\pm$};
\draw[black] (-4,-3.5) node [anchor=west] {\tiny $4m_0=(b_0-1)^2$};
\end{scope}
\end{tikzpicture}
\end{center}
\caption{Collecting the restimates}\label{fig1}
\end{figure}

We want to collect the estimates proved so far. Estimates will be obtained as combination from results of the hyperbolic zone and from the dissipative zone. One is related to the large-time behaviour of all non-zero frequencies, while the other plays a role in estimating the exceptional frequency $\xi=0$. 

\subsubsection{Estimates for the fundamental solution} 
High frequencies are described by a WKB expansion of solutions giving an overall decay estimate based on the function $\lambda(t)$. In Figure~\ref{fig1} this corresponds to the dashed line in the complex plane. The two dots correspond to the exponents $\mu_\pm$ arising from the Levinson's theorem. They are responsible for the small frequency behaviour and the interplay of the relation of these dots and the dashed line will be the major reason for the appearing different cases of final estimates. 

\begin{table}\centering
\begin{tabular}{|c|c|c|c|}
\hline
conditions on $m_0$ and $b_0$ & behaviour in ${\mathcal Z}_{\rm diss}(N)$ & behaviour in ${\mathcal Z}_{\rm hyp}(N)$ \\
\hline
\hline
$(b_0-1)^2<4m_0$ & $(t/s)^{-(b_0+1)/2}$ & $(t/s)^{-b_0/2}$ \\
\hline
$4m_0=(b_0-1)^2$ & $(t/s)^{-(b_0+1)/2+\varepsilon}$ & $(t/s)^{-b_0/2}$   \\
\hline 
$b_0(b_0-2)\le 4m_0<(b_0-1)^2$ & $(t/s)^{\mu_+}$ , $\mu_+\le -b_0/2$ & $(t/s)^{-b_0/2}$   \\
\hline
$0\le 4m_0<b_0(b_0-2)$ & $(t/s)^{\mu_+}$ , $\mu_+>-b_0/2$  & $(t/s)^{-b_0/2}$  \\
\hline
\end{tabular}
\caption{Estimates of fundamental solutions assuming (A1).}\label{table1}
\end{table}

\begin{table}\centering
\begin{tabular}{|c|c|c|c|}
\hline
conditions on $m_0$ and $b_0$ & behaviour in ${\mathcal Z}_{\rm diss}(N)$ & behaviour in ${\mathcal Z}_{\rm hyp}(N)$ \\
\hline
\hline 
$b_0(b_0-2)< 4m_0<(b_0-1)^2$ & $(t/s)^{\mu_+}$ , $\mu_+< -b_0/2$ & $(t/s)^{-b_0/2}$   \\
\hline
$0\le 4m_0<b_0(b_0-2)$ & $(t/s)^{\mu_+}$ , $\mu_+>-b_0/2$  & $(t/s)^{-b_0/2}$  \\
\hline
\end{tabular}
\caption{Estimates of fundamental solutions assuming (A2).}\label{table2}
\end{table}

The main estimates obtained so far can be seen in Tables~\ref{table1} and \ref{table2}.
We have to distinguish between the situation of condition (A1) in Hypothesis 2 and the situation of condition (A2) in Hypothesis 2. In the latter case we can only treat mass terms satisfying $4m_0< (b_0-1)^2$.

\subsubsection{Choice of parameters} The number of diagonalisation steps needed in the hyperbolic zone determines the zone constant $N$ and thus the decomposition of the phase space. When proving energy estimates it will be enough to apply one non-trivial step of diagonalisation in the hyperbolic zone
and for this any choice of $N$ will be good. When proving dispersive estimates several such steps are necessary and $N$ has to be chosen large enough. 

The number $\ell$ of derivatives required in Hypothesis~\ref{Hypo6.1} depends on the number of diagonalisation steps to be used and the needed symbol properties of the matrix function $\mathcal Q_k(t,\theta_{|\xi|},\xi)$. When proving energy estimates, $\ell=1$ is sufficient.


\section{Estimates}\label{sec3}

\subsection{Energy estimates and their sharpness}

The representation of solutions obtained so far allow us to conclude estimates for the solution
and its derivatives. This section is devoted to the study of
estimates, which are directly related to our micro-energy.
\begin{thm} \label{Nestea6}
Assume Hypothesis \ref{Hypo6.1} with $\ell=1$, Hypothesis \ref{Hypo6.2} with $\sigma=1$ and $b_0(b_0-2)\le 4m_0$. Then the $L^2-L^2$ estimate
\begin{equation}
   \| (1+t)^{-1} u(t,\cdot)\|_{L^2} + \| \nabla u(t,\cdot)\|_{L^2} + \| u_t(t,\cdot)\|_{L^2} \lesssim \frac1{\lambda(t)} \big( \|u_0\|_{H^1} + \|u_1\|_{L^2}\big)
\end{equation}
holds true for any solution $u$ of \eqref{KleinGordonDamping} to initial data $u_0\in H^1(\R^n)$ and $u_1\in L^2(\R^n)$.
\end{thm}
\begin{proof}
We first recall the stronger statement
\begin{align}
    \| \mathcal{E}(t,s,\xi)  \| \lesssim \frac{\lambda(s)}{\lambda(t)}
\end{align}
for the fundamental solution $\mathcal E(t,s,\xi)$ constructed in Section~\ref{sec2}. We start by considering the dissipative zone. Here
Theorem~\ref{WanWir} in combination with Remark~\ref{rota98} yields
\begin{align}
    \|\mathcal{E}(t,s,\xi)\| \lesssim \left(\frac{1+t}{1+s}\right)^{\mu_+} \lesssim \frac{\lambda(s)}{\lambda(t)}
\end{align}
uniform with respect to $0\le s\le t\le \theta_{|\xi|}$.

Next we consider the hyperbolic zone and apply Proposition~\ref{judô2} and Theorem~\ref{Lema6} with $k=1$. 
Hence, with uniformly bounded matrices $N_1(t,\xi)$, $N_1(t,\xi)^{-1}$ and $\mathcal{Q}_1(t,s,\xi)$ and the unitary matrix $\mathcal E_0(t,s,\xi)$
we obtain from \eqref{Repres6} 
\begin{equation}
   \mathcal E(t,s,\xi) = \frac{\lambda(s)}{\lambda(t)} M N_1(t,\xi) \mathcal E_0(t,s,\xi) \mathcal Q_1(t,s,\xi) N_1(s,\xi)^{-1} M^{-1}
\end{equation}
and thus the desired norm bound.

Now we show how this implies the desired energy estimate. We first observe that the definition of the micro-energy implies
\begin{equation}
     \|  U(0,\cdot) \|_{L^2} \approx \| u_0 \|_{H^1} + \|u_1\|_{L^2}
\end{equation}
as equivalence of norms (with constants depending on the zone constant $N$). Furthermore,
\begin{equation}
    |\xi|\widehat u(\xi) = \frac{|\xi|}{h(t,\xi)} h(t,\xi) \widehat u(\xi) 
\end{equation}
with $|\xi| / h(t,\xi)$ uniformly bounded by $1$. Similarly
\begin{equation}
    \frac1{1+t} \widehat u(\xi) = \frac{1}{(1+t)h(t,\xi)} h(t,\xi) \widehat u(\xi) 
\end{equation}
and again by the definition of the zone $(1+t)h(t,\xi)\ge N$. 
Therefore, from $U(t,\xi) = \mathcal E(t,0,\xi)U(0,\xi)$ we conclude the desired estimates
\begin{equation}
   \|u_t(t,\cdot)\|_{L^2} \le \| U(t,\cdot)\|_{L^2} \le \frac1{\lambda(t)} \|U(0,\cdot)\|_{L^2}  \lesssim  \frac1{\lambda(t)} \big( \|u_0\|_{H^1} + \|u_1\|_{L^2}\big)
\end{equation}
as well as
\begin{align}
   \| \nabla u(t,\cdot)\|_{L^2}  = \||\xi| \widehat u(t,\cdot)\|_{L^2} \le \| U(t,\cdot)\|_{L^2} & \lesssim  \frac1{\lambda(t)} \big( \|u_0\|_{H^1} + \|u_1\|_{L^2}\big) \\
   \| (1+t)^{-1} u(t,\cdot)\|_{L^2} \le \| U(t,\cdot)\|_{L^2}& \lesssim  \frac1{\lambda(t)} \big( \|u_0\|_{H^1} + \|u_1\|_{L^2}\big)
\end{align}
by the aid of Plancherel's theorem. This completes the proof.
 \end{proof}
 
 If we assume Hypothesis \ref{Hypo6.2} with $\sigma>1$ we have to restrict the admissible values of $m_0$ further. The proof goes in analogy to the above one replacing Theorem~\ref{WanWir} by Theorem~\ref{WanWir2} for the treatment of the dissipative zone.
 
 \begin{thm} \label{Nestea6a}
Assume Hypothesis \ref{Hypo6.1} with $\ell=1$, Hypothesis \ref{Hypo6.2} with $\sigma\in(1,2]$ and $b_0(b_0-2)\le 4m_0<(b_0-1)^2$. Then the $L^2-L^2$ estimate
\begin{equation}
   \| (1+t)^{-1} u(t,\cdot)\|_{L^2} + \| \nabla u(t,\cdot)\|_{L^2} + \| u_t(t,\cdot)\|_{L^2} \lesssim \frac1{\lambda(t)} \big( \|u_0\|_{H^1} + \|u_1\|_{L^2}\big)
\end{equation}
holds true for any solution $u$ of \eqref{KleinGordonDamping} to initial data $u_0\in H^1(\R^n)$ and $u_1\in L^2(\R^n)$.
\end{thm}

We will give some examples to show the applicability of the previous results.
\begin{exam} \label{Exam6.1}
Let us consider for $b_0,m_0\in\R$
\begin{align}
b(t) &= \frac{b_0}{1+t} + \frac{h_1(t)}{1+t}, \\
m(t) &= \frac{m_0}{(1+t)^2} +\frac{h_2(t)}{(1+t)^2}
\end{align}
with uniformly bounded $h_j(t)$, $j=1,2$ and uniformly bounded $t\partial_t h_j(t)$ and with the integrability condition
\begin{equation}\label{eq:61-1}
  \int_0^\infty |h_j(t)| \frac{\d t}{1+t} < \infty,\qquad  j=1,2.
\end{equation}
Then Hypothesis \ref{Hypo6.1} is satisfied with $\ell=1$ and Hypothesis \ref{Hypo6.2} is satisfied with $\sigma=1$.
If we further suppose that $b_0(b_0-2) \le 4m_0$ then the energy estimate
\begin{align}\label{eq:en-est-61}
\| \left( (1+t)^{-1} u(t,\cdot),   u_t(t,\cdot), \nabla_x u(t,\cdot)  \right)  \|_{L^2} \lesssim (1+t)^{-\frac{b_0}{2}}\left( \| u_0 \|_{H^1} + \| u_1 \|_{L^2}\right)
\end{align}
holds true. The decay is independent of $m_0$ and related to the decay for non-effective wave damped models treated in \cite{Wirth1}.
\end{exam}
\begin{exam}\label{Exam6.1a}
We consider the same situation as in the previous example, but replace \eqref{eq:61-1} by
\begin{equation}\label{eq:61-2}
  \int_0^\infty |h_j(t)|^\sigma \frac{\d t}{1+t} < \infty,\qquad  j=1,2,
\end{equation}
then under the more restrictive condition  $b_0(b_0-2) \le  4m_0< (b_0-1)^2$ on the numbers $m_0$ and $b_0$ the {\em same} estimate
\eqref{eq:en-est-61} holds true. To be more specific, this allows to treat
\begin{align}
b(t) &= \frac{b_0}{1+t} + \frac{b_1}{(\e+t)(\ln(\e+t))^\gamma},\\
m(t) &= \frac{m_0}{(1+t)^2} + \frac{m_1}{(\e+t)^2(\ln(\e+t))^\gamma}
\end{align}
with arbitrary $b_1$, $m_1$ and $\gamma\in(1/2,1]$. It satisfies \eqref{eq:61-2} with $\sigma\in (\gamma^{-1},2]$. 
\end{exam}
\begin{rem}
The Examples \ref{Exam6.1} and \ref{Exam6.1a} show us that  small mass terms have no influence on our estimates of the energy of solutions to 
the Cauchy problem \eqref{KleinGordonDamping}.
\end{rem}

The energy estimates of the above theorems are sharp. For this we refer to Theorem~\ref{Lema6} and the uniform invertibility of 
$\mathcal Q_1(t,s,\xi)$ within the hyperbolic zone. We give only a preliminary statement, but it implies the sharpness.

\begin{lem}
Assume the initial data $u_0,u_1\in\mathcal S(\R^n)\setminus\{0\}$ are Schwartz functions such that $\supp \widehat u_0$ and $\supp \widehat u_1$ are compact and
contained in the set $\{ \xi : |\xi|>N\}$. Then the limit
\begin{equation}
   \lim_{t\to\infty} \lambda(t)^2 \big( \| \nabla u(t,\cdot)\|_{L^2}^2 +  \| u_t (t,\cdot) \|_{L^2}^2\big) > 0
\end{equation} 
exists and is different from zero.
\end{lem}
\begin{proof}
As the Fourier support is preserved under evolution, we conclude from Plancherel identity that $\|U(t,\cdot)\|_{L^2}^2 = \| \nabla u(t,\cdot)\|_{L^2}^2 + \|u_t(t,\cdot)\|_{L^2}^2$.
Therefore,  by  Theorem~\ref{Lema6} 
 \begin{equation}
    \lambda(t) \| U(t,\xi)\| =  \| N_1(t,\xi) \mathcal E_0(t,0,\xi) \mathcal Q_1(t,0,\xi) N_1(0,\xi)^{-1} M^{-1} U(0,\xi)\|
 \end{equation}
 and $\| N_1(t,\xi)  - I \| \lesssim (1+t)^{-1}$ uniformly in $\xi\in\supp U(0,\cdot)$. This implies, using the fact that
 $\mathcal E_0$ is unitary,
  \begin{equation}
    \lambda(t) \| U(t,\xi)\| =  \| \mathcal Q_1(t,0,\xi) N_1(0,\xi)^{-1} M^{-1} U(0,\xi)\| + C (1+t)^{-1}. 
 \end{equation}
 The constant $C$ depends on the solution. The first term tends to a limit and thus
   \begin{equation}
  \lim_{t\to\infty}  \lambda(t) \| U(t,\xi)\| =  \| \mathcal Q_1(\infty,0,\xi) N_1(0,\xi) ^{-1} M^{-1} U(0,\xi)\| 
  	 \ge c^{-1} \|U(0,\cdot)\| ,
 \end{equation}
 with $c=\sup_{|\xi|>N} \|M N_1(0,\xi)\mathcal Q_1(\infty,0,\xi)^{-1}\|<\infty$.
 Integrating the square of this inequality with respect to $\xi$ proves the statement.
\end{proof}

\subsection{Large dissipation and improvements by moment conditions} If the dissipation term is large, $4m_0 < b_0(b_0-2)$ then the exponent 
$\mu_+$ from the dissipative zone is larger than $-b_0/2$. Therefore, small frequencies yield slower decay than the hyperbolic zone. This phenomenon characterises the transition from non-effective to effective dissipation. 

Assuming conditions on the initial data we can control this behaviour of small frequencies and still obtain the same decay rates as before. This is obvious, if
$0\not\in\supp U(0,\cdot)$. We want to do better and use moment conditions. We denote
by $L^1_\kappa(\R^n)$ the weighted Lebesgue space
\begin{equation}
  L^1_\kappa(\R^n) = \{ f  \;:\;  \int |f(x)|  \, (1+x^2)^{\kappa/2} \d x <\infty \}
\end{equation}
and assume data have finite energy and belong to such a space.

\begin{thm}\label{thm:33}
Assume Hypothesis \ref{Hypo6.1} with $\ell=1$, Hypothesis \ref{Hypo6.2} with $\sigma\ge 1$ and $4m_0<b_0(b_0-2)$. 
Let further 
\begin{align}\label{eq:kappa}
2\kappa &= 1+\sqrt{(b_0-1)^2-4m_0} \qquad (\sigma=1),\\
2\kappa &> 1+\sqrt{(b_0-1)^2-4m_0} \qquad (\sigma>1).
\end{align}
If  the initial data 
belong to $u_0\in H^1(\R^n)\cap L^1_{\kappa'}(\R^n)$ and $u_1\in L^2(\R^n)\cap L^1_{\kappa'}(\R^n)$ with $[\kappa']>\kappa-\frac n2$
and satisfy the moment conditions
\begin{equation}
   \int x^\alpha u_0(x) \d x = \int x^\alpha u_1(x) \d x = 0
\end{equation}
for all multi-indices $|\alpha|\le \kappa'$ then the  solution satisfies
\begin{equation}\label{eq:3.27}
   \| (1+t)^{-1} u(t,\cdot)\|_{L^2} + \| \nabla u(t,\cdot)\|_{L^2} + \| u_t(t,\cdot)\|_{L^2} \lesssim \frac1{\lambda(t)}.
\end{equation}
\end{thm}  
\begin{proof}
Before giving the details, we outline the main strategy. The $L^1$-condition imposed on the initial data implies continuity in the Fourier image, such that moment conditions determine the order of zero in $\xi=0$. If we can show that $|\xi|^{-\kappa} U(0,\xi)$ is locally square-integrable near $\xi=0$, it suffices to prove the 
estimate
\begin{equation}\label{eq:3.28}
    |\xi|^\kappa \|\mathcal E(t,0,\xi)\| \lesssim \frac1{\lambda(t)}
\end{equation}  
uniformly within $|\xi|\le N$. Then $(1+t)^{-1}\lesssim h(t,x)$ and $|\xi|\lesssim h(t,\xi)$ imply \eqref{eq:3.27}.

{\sl Step 1. Proof of estimate \eqref{eq:3.28}.} First, we consider the case $\sigma=1$. Assume, $(t,\xi)\in\mathcal Z_{\rm diss}(N)$. Then  
Theorem~\ref{WanWir} yields
\begin{equation}
    \| \mathcal E(t,0,\xi)\|\lesssim (1+t)^{\mu_+} 
\end{equation}
with $\mu_+$ given in \eqref{eq:mu}. If $\kappa$ satisfies \eqref{eq:kappa} then 
\begin{equation}
 \kappa>0\qquad\text{and}\qquad   \mu_+ - \kappa = -\frac{b_0}2
\end{equation}
and the definition of the zone implies
\begin{equation}
   |\xi|^\kappa (1+t)^{\mu_+} = |\xi|^\kappa (1+t)^{\kappa} (1+t)^{-b_0/2} \le N^\kappa (1+t)^{-b_0/2}.
\end{equation}
By \eqref{eq:2.6} we know $(1+t)^{b_0/2}\asymp \lambda(t)$ and Remark~\ref{rota98} yields \eqref{eq:3.28} for  all $(t,\xi)\in\mathcal Z_{\rm diss}(N)$.  Furthermore,  Theorem~\ref{Lema6} with $k=1$ extends this estimate to all $(t,\xi)$ with $|\xi|\le N$.

The treatment of $\sigma>1$ is similar. Theorem~\ref{WanWir2} yields for arbitrary $\varepsilon>0$ 
\begin{equation}
    \| \mathcal E(t,0,\xi)\|\lesssim (1+t)^{\mu_++\varepsilon}, 
\end{equation}
such that for sufficiently small $\varepsilon$ again $ |\xi|^\kappa (1+t)^{\mu_+} \le N^{\kappa} (1+t)^{-b_0/2-\varepsilon/2} \lesssim 1/\lambda(t)$.  
This implies \eqref{eq:3.27} within $\mathcal Z_{\rm diss}(N)$ and Theorem~\ref{Lema6} extends this to  $|\xi|\le N$.

{\sl Step 2. Preparation of initial data.} 
The assumption $u_j\in L^1_{\kappa'}(\R^n)$ implies that $\widehat u_j\in C^{[\kappa']}(\R^n)$. Furthermore, by the moment condition
we know that
\begin{equation}
   \partial_\xi^\alpha \widehat u_j(\xi)\big|_{\xi=0} =0,\qquad |\alpha|\le [\kappa']
\end{equation}
and therefore for $|\xi|\le N$
\begin{align}
   |\xi|^{-[\kappa']}  \|U(0,\xi)\| \lesssim 1. 
\end{align}
As $|\xi|^{[\kappa']-\kappa}$ is locally square integrable, the same is true for the function $|\xi|^{-\kappa} U(0,\xi)$ and the desired statement follows.
\end{proof}

\subsection{Modified scattering result}
Now we discuss the sharpness of energy estimates again and formulate a more precise statement. In fact, there is a relation between solutions to the 
Cauchy problem with mass and dissipation
\begin{align} \label{6.32}
u_{tt} - \Delta u + b(t)u_t + m(t) u =0,\qquad u(0,x)=u_0(x),\quad u_t(0,x)=u_1(x),
\end{align}
under our hypotheses and solutions of the free wave equation
\begin{align}\label{6.33}
v_{tt}-\Delta v =0, \qquad v(0,x) = v_0(x), \quad  v_t(0,x)=v_1(x),
\end{align}
with appropriate related data.  We follow some ideas of Wirth \cite{Wirth3} and gives (in combination with the energy conservation for free waves) a very precise description of sharpness of the above energy estimates.
\begin{thm}\label{thm34}
Assume Hypothesis \ref{Hypo6.1} with $\ell=1$ and Hypothesis \ref{Hypo6.2} with 
\begin{equation}
\sigma =1 \qquad \text{and} \qquad b_0(b_0-2)\le 4m_0
\end{equation}
or with
\begin{equation}
\sigma \in(1,2] \qquad \text{and} \qquad b_0(b_0-2)\le 4m_0 < (b_0-1)^2.
\end{equation}
Then there exists a bounded operator 
\begin{equation}
 W_+:  H^1(\R^n) \times L^2(\R^n)  \rightarrow \dot H^1(\R^n) \times L^2(\R^n) 
\end{equation}
 such that
for Cauchy data $(u_0,u_1) \in  H^1(\R^n) \times L^2(\R^n)$ of \eqref{6.32} and associated data
$(v_0,v_1)= W_+(u_0,u_1) \in H^1(\R^n) \times L^2(\R^n) $ to \eqref{6.33} the corresponding solutions $u = u(t,x)$
and $v= v(t,x)$ satisfy
\begin{align}
\| \lambda(t)u_t(t,\cdot) - v_t(t,\cdot)\|_{L^2} & \to 0,\\
\| \lambda(t) \nabla u(t,\cdot) - \nabla v(t,\cdot)\|_{L^2}  &\rightarrow 0,
\end{align}
as $t\rightarrow \infty$.
\end{thm}
\begin{proof}
We sketch the major steps of the proof. First let for $\varepsilon>0$  
\begin{equation}
 F_\varepsilon = \left\{ U_0\in L^2(\mathbb R^n)\times L^2(\mathbb R^n) : \ \widehat{U}_0(\xi)=0 \ \text{for any~$|\xi|\leq\varepsilon$} \right\}.
\end{equation}
Then the union $\mathcal L = \bigcup_{\varepsilon>0} F_\varepsilon$ is dense in $L^2(\mathbb R^n)\times L^2(\mathbb R^n)$. We construct the wave operator
$W_+$ as pointwise limit on $\mathcal L$ and apply Banach--Steinhaus theorem to show strong convergence on the energy space. 

We start by introducing notation. Let $\mathcal{E}_0(t,s,\xi)$ be the free propagator from  
\eqref{neve}. Then for any solution $v$ to the free wave equation \eqref{6.33} the classical micro-energy 
$\widetilde V(t,\xi) = \left( |\xi| \widehat{v}, \D_t \widehat{v} \right)^T$ satisfies $\widetilde V(t,\xi) = \mathcal E_{\rm fr}(t,s,\xi) \widetilde V(s,\xi)$ 
for 
\begin{equation}
   \mathcal E_{\rm fr}(t,s,\xi) = M \mathcal E_0(t,s,\xi) M^{-1}.
\end{equation}
We compare this to the propagator $\mathcal E(t,s,\xi)$ constructed in Section~\ref{sec2} for the micro-energy $U(t,\xi) = \left( h(t,\xi) \widehat{u}, \D_t \widehat{u} \right)^T$ associated to solutions $u$ of \eqref{6.32}. By \eqref{Repres6} and \eqref{Repres6s} with $k=1$ this propagator is a product of known matrix functions 
$\mathcal E_0(t,s,\xi)$, $\mathcal Q_1(t,s,\xi)$, $N_1(t,\xi)$ and the function $\lambda(t)$ with known behaviour. We recall from Theorem~\ref{Lema6} that
\begin{equation}
  \lim_{t\to\infty} \mathcal Q_1(t,s,\xi) = \mathcal Q_1(\infty,s,\xi)
\end{equation}
holds true locally uniform in $s$ and $\xi$. Furthermore, $N_1-\mathrm I \in \mathcal S^{\ell}\{-1,1\}$. Our first aim is to show that the limit
\begin{align}
W_+(\xi) = \lim_{t\rightarrow \infty} \lambda(t) \mathcal{E}_{\rm fr}(t,0,\xi)^{-1} \mathcal{E}(t,0,\xi)
\end{align}
exists uniformly in $|\xi| >\varepsilon$ and thus as norm-limit on $F_\epsilon$. Indeed, by \eqref{Repres6} and \eqref{Repres6s} 
\begin{multline}
    \lambda(t) \mathcal{E}_{\rm fr}(t,0,\xi)^{-1} \mathcal{E}(t,0,\xi) \\
           = \lambda(\theta_{|\xi|}) M \mathcal{E}_0(0,t,\xi) N_1(t,\xi) \mathcal{E}_0(t,\theta_{|\xi|},\xi) \\\times \mathcal{Q}_1(t,\theta_{|\xi|},\xi)
                        N_1(\theta_{|\xi|},\xi)^{-1}M^{-1}\mathcal{E}(\theta_{|\xi|},0,\xi)
\end{multline}
such that using 
\begin{multline}
\mathcal{E}_0(0,t,\xi) N_1(t,\xi) \mathcal{E}_0(t,\theta_{|\xi|},\xi) \\= \mathcal{E}_0(0,\theta_{|\xi|},\xi) +
\mathcal{E}_0(0,t,\xi) ( N_1(t,\xi) -I)\mathcal{E}_0(t,\theta_{|\xi|},\xi)
\end{multline}
as well as $N_k(t,\xi) \rightarrow I$ uniformly on $|\xi| \geq \varepsilon$ implies
\begin{multline}
   \lim_{t\to\infty}  \lambda(t) \mathcal{E}_{\rm fr}(t,0,\xi) \mathcal{E}(t,0,\xi) \\
   =  \lambda(\theta_{|\xi|}) M\mathcal{E}_0(0,\theta_{|\xi|},\xi) \mathcal{Q}_1(\infty,\theta_{|\xi|},\xi)
                        N_1(\theta_{|\xi|},\xi)^{-1} M^{-1}\mathcal{E}(\theta_{|\xi|},0,\xi)
\end{multline}
uniformly on $|\xi|>\varepsilon$. We denote this limit as $W_+(\xi)$. Due to the energy estimates of Theorem \ref{Nestea6} we know that
$\lambda(t)\mathcal{E}_{\rm fr}(t,0,\xi)^{-1} \mathcal{E}(t,0,\xi) $ is uniformly bounded in $t$ and $\xi$. Therefore, the matrix $W_+(\xi)$ is uniformly
bounded in $\xi$. Furthermore, by Banach--Steinhaus theorem
we know that 
\begin{equation}\label{eq:3.49}
   W_+(\D) = \slim_{t\to\infty} \lambda(t)\mathcal{E}_{\rm fr}(t,0,\D)^{-1} \mathcal{E}(t,0,\D) 
\end{equation}
exists as strong limit in $L^2(\R^n)\times L^2(\R^n)$.

Now we relate the initial data by $V_0 = W_+(\xi) U_0$. Then the difference of the corresponding micro-energies satisfies
\begin{multline}
\lambda(t)U(t,\xi) - V(t,\xi) 
                    =  \mathcal{E}_{\rm fr} (t,0,\xi)\left(\lambda(t) \mathcal{E}_{\rm fr}(t,0,\xi)^{-1}  \mathcal{E}(t,0,\xi) - W_+(\xi) \right)U(0,\xi).
\end{multline}
By the strong convergence \eqref{eq:3.49} combined with the fact that the free propagator is unitary the limit behaviour
\begin{align}
&  \|  \lambda(t) \nabla u(t,\cdot) - \nabla v(t,\cdot) \|_{L^2}\to 0,\\
 & \|  \lambda(t) u_t (t,\cdot) - v_t(t,\cdot) \|_{L^2} \to 0
\end{align}
follows for all inital data from $H^1(\R^n)\times L^2(\R^n)$. This completes the proof.
\end{proof}

\begin{rem}
The modified scattering result involves only the hyperbolic energy terms $\nabla u(t,\cdot)$ and $u_t(t,\cdot)$. If we are interested in results containing also the solution $u(t,\cdot)$ itself, we can not hope for the same kind of (unweighted) result. Note for this, that the estimate
$\|v(t,\cdot)\|_{L^2} \le t (\|v_0\|_{L^2} + \|v_1\|_{H^{-1}})$ is in general sharp for solutions to the Cauchy problem for the free wave equation, nevertheless there are no
initial data with this precise rate. We only have $\|v(t,\cdot)\|_{L^2} = o(t)$ as $t\to\infty$ for each (fixed) solution. Similarly one obtains for solutions to  \eqref{6.32}
to initial data from $L^2(\R^n)\times H^{-1}(\R^n)$
\begin{equation}
   \lim_{t\to\infty}  \frac{\lambda(t)}{1+t} \|u(t,\cdot)\|_{L^2} = 0.
\end{equation}
This rate is sharp for general data and can only by improved by further assumptions on initial data. We omit the proof.
\end{rem}


\subsection{$L^p$--$L^q$ estimates} Finally, we want to give dispersive type estimates for solutions. These are $L^p$--$L^q$ estimates for conjugate Lebesgue indices. The estimate is again independent of $m_0$, but the range of admissible $b_0$ depends on $m_0$. For this statement we need to use the representations fo Section~2 with $k>1$ and, therefore, we also need higher regularity of the coefficient functions compared to the energy estimates given before.

\begin{thm} \label{NesteaLpLq}
Assume Hypothesis \ref{Hypo6.1} with $\ell=n+1$,
Hypothesis \ref{Hypo6.2} with 
\begin{equation}
\sigma =1 \qquad \text{and} \qquad b_0(b_0-2)\le 4m_0
\end{equation}
or with
\begin{equation}
\sigma \in(1,2] \qquad \text{and} \qquad b_0(b_0-2)\le 4m_0 < (b_0-1)^2.
\end{equation}
Then the $L^p-L^q$ estimate
\begin{multline}
\| (1+t)^{-1} u(t,\cdot) \|_{L^q} 
+ \|\nabla u(t,\cdot)\|_{L^q}
+\|u_t(t,\cdot)\|_{L^q} \\
\le C_{p,n} 
\frac{1}{\lambda(t)} (1+t)^{-\frac{n-1}{2}\left( \frac{1}{p} - \frac{1}{q} \right)}
\big( \|u_0\|_{W^{r+1,p}} + \|u_1\|_{W^{r,p}}\big)
\end{multline}
holds true for $p \in (1,2],$ $pq=p+q$ and with Sobolev regularity $r = n\left(  \frac{1}{p} - \frac{1}{q} \right).$
\end{thm}
\begin{proof}
The proof is divided into two steps, we give estimates separately for the dissipative and the hyperbolic zone of the phase space. In the dissipative zone the estimate follows by a simple argument using H\"older inequality, while for the hyperbolic zone we have to employ the stationary phase method. The latter is done by reducing the estimate to the well-known estimate for the free wave equation. 

{\sl Step 1. Considerations in the dissipative zone.}
We first recall the estimate
\begin{equation}
\|  \mathcal{E} (t,0,\xi)  \varphi_{\rm diss} (t,\xi) \| \lesssim \frac{1}{\lambda(t)}
\end{equation}
obtained in Section~\ref{sec2}. If the initial data belong to Sobolev spaces over $L^p$ the initial micro-energy satisfies $U_0\in L^{q}\{|\xi|\le N\}$
and therefore
\begin{align}
\left\|   \mathcal F^{-1} \left( \mathcal{E} (t,0, \cdot ) \varphi_{\rm diss} (t, \cdot ) U_0   \right)
\right\|_{L^q} &\leq  \|  \mathcal{E} (t,0, \cdot ) \varphi_{\rm diss} (t, \cdot ) U_0 \|_{L^p}\notag \\
                                             & \leq   \|  \mathcal{E} (t,0, \cdot )   \|_{L^\infty} \|  \varphi_{\rm diss} (t, \cdot )
                                             \|_{L^{\frac{pq}{q-p}}} \| U_0 \|_{L^q}\notag \\
                                                   &                                   \lesssim   \frac{1}{\lambda(t)}
                                                                                          (1+t)^{-n\left( \frac{1}{p}-\frac{1}{q}\right)}\| U_0 \|_{L^q},
    \end{align}
based on 
\begin{equation}
\|\varphi_{\rm diss}(t,\cdot)\|_{L^{\frac{pq}{q-p}}} \lesssim \left(\int_{|\xi|\le N (1+t)^{-1}} \d\xi \right)^{\frac1p-\frac1q}
\lesssim (1+t)^{-n(\frac1p-\frac1q)},
    \end{equation}
which is a better decay compare to the statement of the theorem.

{\sl Step 2. Considerations in the hyperbolic zone.} For large  frequencies we use the representation of \eqref{Repres6} 
to split the propagator into several parts and estimate each of them separately. For this we choose $k$ such that
\begin{equation}
\ell = 2(k-1) \qquad\text{and}\qquad k-1 \geq \left\lceil \frac{n}{2} \right\rceil.
\end{equation}  We use the short-hand notation 
$p,r\to p,r$ to denote operators acting between the Sobolev / Bessel potential spaces $H^{r,p}  \to H^{r,p}$ 
of regularity $r$ over $L^p$. Then $ \mathcal{E}(t,0,\D)\varphi_{\rm hyp}^\ell ( \D) $ equals
\begin{equation}
   \frac{1}{\lambda(t)} \underbrace{M N_k(t,\D)}_{q \rightarrow q}
\underbrace{\mathcal{E}_0(t,0,\D)}_{p,r \rightarrow q}\underbrace{\mathcal{Q}_k(t,0,\D)}_{p,r \rightarrow p,r}
\underbrace{N_k(0,\D)^{-1}M^{-1}}_{p,r \rightarrow p,r}\underbrace{\varphi_{\rm hyp}^\ell ( \D)}_{p,r\to p,r}
\end{equation}
and we estimate each factor as indicated. To estimate operators for fixed $p$ (and $r$), we apply the H\"ormander--Mikhlin Theorem.
Indeed, by Proposition~\ref{judô2} we know that $MN_k(t,\xi) \in \mathcal S_N^{\ell - k +1}\{0,0\}$. Therefore,  
 $MN_k(t,\xi) \in \dot{S}^0_{\ell-k+1}$ uniformly in $t$ and by H\"ormander--Mikhlin Theorem we conclude
 that 
 \begin{equation}\label{eq:MNk-est}
 \|MN_k(t,\D)\|_{q\to q}\le C
 \end{equation}
 uniformly in $t$. Here it is essential that $\ell -k +1 \geq \left\lceil \frac{n}{2} \right\rceil $.
             
 Next, the well-known dispersive estimate for the free wave equation is equivalent to
 \begin{equation}\label{eq:E0estfree}
    \| \mathcal E_0(t,0,\D) \|_{p,r\to q} \le C (1+t)^{-\frac{n-1}2(\frac1p-\frac1q)}.
 \end{equation}         

For the remaining factors we observe that a Fourier multiplier is bounded between Bessel potential spaces of order $r$
if and only if it is bounded on the $L^p$-spaces. Therefore, it is again sufficient to apply the H\"ormander--Mikhlin multiplier theorem.
By Theorem~\ref{Lema6} we know that $\mathcal{Q}_k(t,0,\xi) \in  \dot{S}_{k-1}^0$ uniformly with respect to $t$ and therefore
\begin{equation}
  \|\mathcal{Q}_k(t,0,\D)\|_{p,r\to p,r} < C
  \end{equation}
uniformly in $t$. Again it is essential that, $k-1 \geq \left\lceil \frac{n}{2} \right\rceil$. Finally, $N_k(0,\xi)^{-1}\in \dot{S}_{k-1}^0$ by construction and 
 $\varphi_{\rm hyp}^\ell \in \dot S^0$. Therefore, it follows that
 \begin{equation}
   \| \mathcal E(t,0,\D)\varphi_{\rm hyp}^\ell(\D) \|_{p,r\to q} \le C \frac{1}{\lambda(t)} (1+t)^{-\frac{n-1}2(\frac1p-\frac1q)}.
 \end{equation}
For the remaining small frequencies we proceed in a similar way.  We have by \eqref{Repres6s} that $\mathcal{E}(t,0,\D)\varphi_{\rm hyp}^s (t, \D)$ equals
\begin{multline}
   \frac{1}{\lambda(t)}  \underbrace{M N_k(t,\D)}_{q \rightarrow q}
\underbrace{\mathcal{E}_0(t,\theta_{|\D|},\D)}_{p,r \rightarrow q}\underbrace{\mathcal{Q}_k(t,\theta_{|\D|},\D)}_{p,r \rightarrow p,r} \underbrace{N_k(\theta_{|\D|},\D)^{-1}M^{-1}}_{p,r \rightarrow p,r} \\
\quad\times \underbrace{ \lambda(\theta_{|\D|}) \mathcal{E}(\theta_{|D|},0,\D) }_{p,r \rightarrow p,r}\underbrace{\varphi_{\rm hyp}^s (t, \D)}_{p,r\to p,r}
\end{multline}
and each of the appearing operators can be again estimated separately. The estimate \eqref{eq:MNk-est} for $M^{-1} N_k^{-1}(t,\D)$ follows in analogy. 
Furthermore, we rewrite $\mathcal{E}_0(t,\theta_{|\D|}, \D) =\mathcal{E}_0(t,0, \D) \mathcal{E}_0(0,\theta_{|\D|}, \D)$. Then 
$\mathcal{E}_0(0,\theta_{|\xi|}, \xi)\varphi_{\rm hyp}^s(t,\xi) \in \dot{S}^0$  implies that the first factor is $L^q$-bounded and therefore the dispersive estimate \eqref{eq:E0estfree} for free waves yields
\begin{equation}
  \|\mathcal{E}_0(t,\theta_{|D|}, D)\varphi_{\rm hyp}^s(t,\D) \|_{p,r\to q} \le C (1+t)^{ -\frac{n-1}2(\frac1p-\frac1q)}.
\end{equation}
By Theorem~\ref{Lema6} it follows that $\mathcal{Q}_k(t,\theta_{|\xi|},\xi) \in \dot{S}_{k-1}^0$. Furthermore, by Proposition~\ref{judô2} and the properties
of $\theta_{|\xi|}$ we know that $N_k(\theta_{|\xi|},\xi) \in \dot S^0$. Therefore 
\begin{equation}
  \|\mathcal{Q}_k(t,\theta_{|\D|},\D)\|_{p,r\to p,r}\le C
  \end{equation}
uniformly in $t$ and
\begin{equation}
  \| N_k(\theta_{|\D|},\D)^{-1} M^{-1}\|_{p,r\to p,r} \le C.
\end{equation} 
By Remark~\ref{defSymbClasMT} we also know that
$ \lambda(\theta_{|\xi|}) \mathcal{E}(\theta_{|\xi|},0,\xi) \in \dot{S}_{l-1}^0$ such that
\begin{equation}
   \| \lambda(\theta_{|\D|}) \mathcal{E}(\theta_{|D|},0,\D)  \|_{p,r\to p,r} \le C.
\end{equation}
Hence, it follows that
 \begin{equation}
   \| \mathcal E(t,0,\D)\varphi_{\rm hyp}^s(\D) \|_{p,r\to q} \le C \frac{1}{\lambda(t)} (1+t)^{-\frac{n-1}2(\frac1p-\frac1q)}
 \end{equation}
 and combining all three parts the statement follows from the definition of the micro-energy.
\end{proof}

\section{Concluding remarks}
{\bf 1.} First, we will give some comments on the relation of the presented results to the known treatments of Wirth \cite{Wirth1} for non-effectively damped wave equations (i.e., $m(t)=0$) and B\"ohme--Reissig \cite{Boehme} for Klein--Gordon equations with non-effective time-dependent mass (i.e., $b(t)=0$). 

If $m_0=0$ and $b_0\in [0,1)\cup(1,2)$ then we are in the setting of \cite{Wirth1} (or \cite{Wirth4} for the particular case $b(t)= \frac{b_0}{1+t}$)
and the estimates of Theorems~\ref{Nestea6} and \ref{NesteaLpLq}, both with $\sigma=1$, reduce to results from these papers. 

If $b_0=0$ we can treat arbitrary $m_0$ and obtain form Theorems~\ref{Nestea6} and \ref{NesteaLpLq}, both with $\sigma=1$,  uniform bounds on the energy as well as the standard wave type $L^p$--$L^q$ decay estimates. The scale-invariant case was considered in \cite{Boehme} with similar observations. 

The results based on Hypothesis~\ref{Hypo6.2} with $\sigma>1$ are new for both situations.

{\bf 2.} The main reason for the present note is to provide a more systematic approach to the small-frequency zone and to highlight its connection to differential equations of Fuchs type. Seen in this light, equations with mass and dissipation are just a model case to set the scene for more general considerations. An extension of this method to give a more systematic small frequency counterpart to \cite{RW11} is planned. This also explains the appendix collecting the asymptotic integration statements in general form.

{\bf 3.} The restriction of Assumption (A2) to the range $\sigma\in(1,2]$ is due to just one application of the Hartmann-Wintner transform of Theorem~\ref{HartWintner}. Applying finitely many such transformations in an iterative way allows to extend Hypothesis~\ref{Hypo6.2}  to arbitrary $\sigma>1$. The price to pay for this is a series of correction terms in \eqref{AsymptBehaPDp} instead of just one. The estimates of theorems~\ref{Nestea6a}, \ref{thm:33}, \ref{thm34} and \ref{NesteaLpLq} still depend on the hyperbolic zone and (as long as the right-hand side of\eqref{AsymptBehaPDp} with $s=0$ is still majorised by $\lambda(t)^{-1}$) are valid unchanged. 

{\bf 4.} The estimates of the solution $u(t,\cdot)$ itself following from Theorems~\ref{Nestea6}, \ref{Nestea6a} and \ref{thm:33} are not optimal in the present form. This is due to the attempted $\sigma$-independent formulation of the results. Under Hypothesis~\ref{Hypo6.2} with $\sigma=1$ it is easily seen that the estimate for the solution can be improved to 
\begin{equation}
   \| u(t,\cdot)\|_{L^2} \lesssim (1+t)^{1+\Re \mu_+}   \big( \|u_0\|_{H^1} + \|u_1\|_{L^2} \big)
\end{equation}
and for the convenience of the reader we give the essential argument behind this improvement. The improvement is based on \eqref{eq:2.6}. Within $\mathcal Z_{\rm diss}(N)$ the construction gave the estimate
\begin{equation} 
 \left|  \frac1{1+t} \widehat u(t,\xi) \right|  \lesssim (1+t)^{\Re \mu_+} \big( \widehat u_0(t,\xi) + \widehat u_1(t,\xi) \big).
\end{equation}
If we consider the hyperbolic zone $\mathcal Z_{\rm hyp} (N)$, we obtain in analogy
\begin{equation} 
  \big|  |\xi| \widehat u(t,\xi) \big|  \lesssim \begin{cases} (1+t)^{-\frac{b_0}2} \big( |\xi| \widehat u_0(t,\xi) + \widehat u_1(t,\xi) \big) \qquad & |\xi|>N\\
 \left( \frac{1+t}{1+t_\xi}\right)^{-\frac{b_0}2} (1+t_\xi)^{\Re \mu_+} \big( \widehat u_0(t,\xi) + \widehat u_1(t,\xi) \big) &|\xi|\le N
  \end{cases}
\end{equation}
and together with $|\xi|(1+t_\xi)=N$, the positivity of $\delta = 1+\frac{b_0}2 +\Re \mu_+ > 0$ and the resulting monotonicity of $t^{\delta}$ the desired estimate follows.
For $\sigma>1$ the latter need not be valid any more and improvements depend on the behaviour of the quotient of the right-hand side of \eqref{AsymptBehaPDp} and $\lambda(t)$.

\begin{appendix}
\section{Asymptotic integration lemmata}\label{secA}
In this appendix we collect some theorems on the asymptotic integration of ordinary differential equations, which are particularly useful for the treatment of the dissipative zone. We formulate them in more general form than used in the present paper. They follow \cite[Sections 1.3 and 1.4]{Eastham} adapted to systems of Fuchs type.

\subsection{Levinson type theorems}

We consider the following system of ordinary differential equations
\begin{equation}\label{eq:A1}
   t\partial_t V (t,\nu) = \big( D(t,\nu) + R(t,\nu) \big) V(t,\nu),\qquad t\ge 1,
\end{equation}
depending on a parameter $\nu\in\Upsilon$. The matrix 
\begin{equation}
D(t,\nu)=\diag\big(\mu_1(t,\nu),\ldots,\mu_d(t,\nu)\big)
\end{equation} 
is diagonal and $R(t,\nu)\in\C^{d\times d}$ denotes a remainder term. 

Under a dichotomy condition imposed on $D$ and appropriate smallness conditions on the remainder $R$ the diagonal matrix $D$ determines asymptotic properties of solutions to
\eqref{eq:A1}. We denote by $e_k$ the $k$-th basis vector of $\C^d$.

\begin{thm}\label{thm:Lev1}
Assume that for $i\ne j$
\begin{multline}\label{eq:LevCond1}
\limsup_{t\to\infty} \sup_{\nu\in\Upsilon} \Re   \int_1^t \big( \mu_i(s,\nu) - \mu_j(s,\nu) \big)\frac{\d s}s <  +\infty
\\  \text{or}\quad
\liminf_{t\to\infty} \inf_{\nu\in\Upsilon} \Re   \int_1^t \big( \mu_i(s,\nu) - \mu_j(s,\nu) \big)\frac{\d s}s > -\infty
\end{multline}
together with
\begin{equation}\label{eq:LevCond2}
    \sup_{\nu\in\Upsilon}   \int_1^\infty \|R(t,\nu)\| \frac{\d t}{t} < \infty.
\end{equation}
Then there exist solutions $V_k(t,\nu)$ to \eqref{eq:A1} satisfying
\begin{equation}
    V_k(t,\nu) = \big( e_k + o(1) \big) \exp\left( \int_{1}^t \mu_k(\tau,\nu) \frac{\d\tau}\tau \right)  
\end{equation}
uniformly in the parameter $\nu\in\Upsilon$.
\end{thm}
\begin{proof}
This is a reformulation of Theorem~1.3.1 from \cite{Eastham} with the substitution $t=\e^x$. For the convenience of the reader we sketch the main idea of the proof. We may replace the dichotomy condition \eqref{eq:LevCond1} by an 'either-or' statement assuming in the first case that in addition
\begin{equation}\label{eq:LevCond1+}
\liminf_{t\to\infty} \inf_{\nu\in\Upsilon} \Re   \int_1^t \big( \mu_i(s,\nu) - \mu_j(s,\nu) \big)\frac{\d s}s = -\infty
\end{equation}
holds true. This yields an ordering of the diagonal entries according to their strength and we may assume without loss of generality that for $i<j$ the first alternative holds true. 
Furthermore, if we write
\begin{equation}
   V(t,\nu)  = Z(t,\nu)  \exp\left( \int_{1}^t \mu_k(\tau,\nu) \frac{\d\tau}\tau \right)  
\end{equation}
for a fixed index $k$ then the function $Z(t,\nu)$ satisfies the transformed equation
\begin{equation}\label{z-eq}
   t\partial_t Z(t,\nu) =  (D(t,\nu)-\mu_k(t,\nu) I + R(t,\nu)) Z(t,\nu)
\end{equation}
and we have to show that there exists a solution to that equation tending to $e_k$ uniformly with respect to $\nu\in\Upsilon$. 
Thus it is sufficient to prove the original theorem for the case $\mu_k=0$. 
Let $\Phi(t) = \Phi_-(t,\nu) + \Phi_+(t,\nu)$ 
be the fundamental solution to the diagonal part, split as
\begin{equation}
\Phi_-(t,\nu) = \diag(\exp\left( \int_{1}^t \mu_1(\tau,\nu) \frac{\d\tau}\tau \right)  ,
\ldots, \exp\left( \int_{1}^t \mu_{k-1}(\tau,\nu) \frac{\d\tau}\tau \right)  ,0,\ldots)
\end{equation}
and 
\begin{equation}
\Phi_+(t,\nu) = \diag(0,\ldots,0,1,\exp\left( \int_{1}^t \mu_{k+1}(\tau,\nu) \frac{\d\tau}\tau \right)  ,\ldots, \exp\left( \int_{1}^t \mu_d(\tau,\nu) \frac{\d\tau}\tau \right)   )
\end{equation}
according to the asymptotics of the entries. Then \eqref{eq:A1} rewrites as an integral equation
\begin{multline}\label{eq:A6}
  V(t,\nu) = e_k + \Phi_-(t,\nu) \int_{t_0}^t \Phi^{-1}(\tau,\nu) R(\tau,\nu) V(\tau,\nu) \frac{\d\tau}{\tau} \\
  - \Phi_+(t,\nu) \int_t^\infty \Phi^{-1}(\tau,\nu) R(\tau,\nu) V(\tau,\nu) \frac{\d\tau}{\tau}.
\end{multline}
By construction we obtain $\|\Phi_-(t,\nu) \Phi(\tau,\nu)^{-1}\|\le C_-$ uniformly on $1\le \tau\le t$ and
$\|\Phi_+(t,\nu)\Phi(\tau,\nu)^{-1}\|\le C_+$ uniformly on $t\le \tau<\infty$. Thus,
this equation can be  solved uniquely in $L^\infty([1,\infty))$ by the contraction mapping principle as
\begin{multline}
 \bigg| \Phi_-(t,\nu) \int_1^t \Phi^{-1}(\tau,\nu) R(\tau,\nu) V(\tau,\nu) \frac{\d\tau}{\tau} \\
  - \Phi_+(t,\nu) \int_t^\infty \Phi^{-1}(\tau,\nu) R(\tau,\nu) V(\tau,\nu) \frac{\d\tau}{\tau} \bigg| 
  \\\le  (C_-+C_+) \int_{t_0}^\infty \|R(\tau,\nu)\|\frac{\d\tau}{\tau} \|V(\cdot,\nu)\|_\infty
\end{multline}
is contractive for $t_0$ sufficiently large. Thus, solutions to \eqref{eq:A6} are uniformly bounded. To show that 
they tend to $e_k$ for $t\to\infty$ uniformly with respect to $\nu\in\Upsilon$ one uses the stronger form \eqref{eq:LevCond1}--\eqref{eq:LevCond1+} of the dichotomy condition. Indeed, writing \eqref{eq:A6} for $t>T$ as
\begin{equation}
  V(t,\nu) = e_k + \Phi_-(t,\nu) \int_{t_0}^T \Phi^{-1}(\tau,\nu) R(\tau,\nu) V(\tau,\nu) \frac{\d\tau}{\tau} + \Psi(t,\nu) 
\end{equation}
with
\begin{multline}
   \Psi(t,\nu) = \Phi_-(t,\nu) \int_{T}^t \Phi^{-1}(\tau,\nu) R(\tau,\nu) V(\tau,\nu) \frac{\d\tau}{\tau} \\
  - \Phi_+(t,\nu) \int_t^\infty \Phi^{-1}(\tau,\nu) R(\tau,\nu) V(\tau,\nu) \frac{\d\tau}{\tau}
\end{multline}
we obtain
\begin{equation}
   \| \Psi(t,\nu)\| \le (C_-+C_+) \int_T^\infty \| R(\tau,\nu)\| \frac{\d\tau}{\tau}  \|V(\cdot,\nu)\|_\infty 
\end{equation}
uniformly in $t\ge T$ and $\nu\in\Upsilon$. Hence, we can choose $T$ large enough such that $\|\Psi(t,\nu)\|\le \varepsilon$. But then the dichotomy condition implies $\Phi_-(t,\nu)\to 0$ uniformly in $\nu$ and thus 
\begin{equation}
   \| V(t,\nu) - e_k \| \le 2\varepsilon
\end{equation}
holds true uniformly in $\nu\in\Upsilon$ and $t>T$ sufficiently large. As $\varepsilon$ was arbitrary, the statement is proven.
\end{proof}

\begin{rem} \label{ConsFundSol}
We will use a special form of the previous theorem, where the diagonal matrices $D$ are constant and independent of $\nu$,
\begin{equation}
   D = \diag (\mu_1, \ldots, \mu_d).
\end{equation}
In this case the dichotomy condition  \eqref{eq:LevCond1} is trivially satisfied as the appearing integrals are all logarithmic functions in $t$
which can't approach  both infinities. Hence, \eqref{eq:LevCond2} is sufficient to conclude the existence of solutions 
\begin{equation}
   V_k(t,\nu) = (e_k + o(1) ) t^{\mu_k} 
\end{equation}
for all $k$ and if {\em in addition} it is known that $\mu_i\ne\mu_j$ for $i\ne j$ this yields a fundamental system of solutions. If the diagonal entries coincide, one has to make further assumptions on lower order terms to get precise asymptotic properties, in particular \eqref{eq:LevCond2} has to be replaced by adding logarithmic terms. 

Levinson's theorem yields a corresponding statement for the fundamental matrix-valued solution to \eqref{eq:A1}. This follows immediately from the following variant of Liouville theorem. We assume for simplicity that $D$ is constant and that the entries are distinct. Then we take the solutions $V_k$ constructed above as fundamental system. Their Wronskian satisfies
\begin{equation}
   \mathcal W_{V_1,\ldots, V_d} (t) = \det\big( V_1(t,\nu) | \cdots | V_d(t,\nu)\big) = t^{\mu_1+\mu_2+\cdots+\mu_k}.
\end{equation}
If we denote by $\mathcal E_V(t,1,\nu)$ the matrix valued solution to
\begin{equation}\label{eq:A2}
   t\partial_t \mathcal E_V (t,1,\nu) = \big( D + R(t,\nu) \big) \mathcal E_V(t,1,\nu),\qquad t\ge 1,
\end{equation}
combined with $\mathcal E_V(1,1,\nu) = \mathrm I$, it follows that
\begin{equation}
  \mathcal E_V(t,1,\nu) = \big( V_1(t,\nu)  | \cdots | V_d(t,\nu) \big)   \big( V_1(1,\nu)  | \cdots | V_d(1,\nu) \big)^{-1}
\end{equation}
and the norm of the inverse matrix can be estimated by Carmer's rule combined with Hadamard's inequality as
\begin{equation}
\|\big( V_1(1,\nu)  | \cdots | V_d(1,\nu) \big)^{-1} \| \le 
d \big(\max_{1\le k\le d} \|V_k(1,\nu)\| \big)^{d-1}
\end{equation}
and thus 
\begin{equation}
   \| \mathcal E_V(t,1,\nu) \| \le C t^{\max_j \Re\mu_j}
\end{equation}
uniformly in $\nu$.
\end{rem}

\begin{rem}\label{rem:scaling}
We can use scaling properties of Fuchs type equations. If $V(t,\nu)$ solves \eqref{eq:A1} then $\tilde V(t,\nu)=V(\lambda t,\nu)$ solves
the rescaled equation
\begin{equation}
   t\partial_t \tilde V(t,\nu) = \big(D(\lambda t,\nu)  + R(\lambda t,\nu)\big) \tilde V(t,\nu).
\end{equation} 
If $\lambda>1$ then 
\begin{equation}
   \int_1^\infty \| R(\lambda t,\nu)\|\frac{\d t}{t} =  \int_\lambda^\infty \| R( t,\nu)\|\frac{\d t}{t} \le  \int_1^\infty \| R( t,\nu)\|\frac{\d t}{t}
\end{equation}
and similarly for the integrals in \eqref{eq:LevCond1}. Hence, the conditions of Levinson's theorem are uniform in $\lambda$ and thus are the constructed 
solutions. Therefore, any estimate of the 
fundamental solution given in Remark~\ref{ConsFundSol} is also uniform and therefore of the form
\begin{equation}
   \|\mathcal E_V(\lambda t,\lambda,\nu)\| =   \|\mathcal E_{\tilde V} (t,1,\nu) \| \le C t^{\max_j \Re\mu_j}
\end{equation}
uniformly in $\lambda>1$ and $\nu\in\Upsilon$.
\end{rem}

\subsection{Hartman--Wintner type theorems}
Now we discuss improvements of Theorem \ref{thm:Lev1} based on a diagonalisation procedure. They allow to handle remainders satisfying
\begin{equation}\label{eq:HWcond}
      \int_1^\infty \| R(t,\nu)\|^\sigma \frac{\d t}{t} < C
\end{equation}   
for some constant $1<\sigma<\infty$. They are constructive and give precise asymptotics similar to the above theorem. We formulate it in more general form with diagonal matrix $D(t,\nu)$ with entries satisfying the stronger form of the dichotomy condition
\begin{multline}\label{eq:LevC2}
   \Re \big( \mu_i(t,\nu) - \mu_j(t,\nu) \big) \le C_-  
\quad \text{or}\quad
  \Re \big( \mu_i(t,\nu) - \mu_j(t,\nu) \big) \ge C_+
\end{multline}
uniform in $t\ge t_0$ and $\nu\in\Upsilon$. It implies \eqref{eq:LevCond1}.

\begin{thm}\label{HartWintner}
Assume \eqref{eq:LevC2} in combination with \eqref{eq:HWcond}. Let further 
\begin{equation}
F(t,\nu) = \diag R(t,\nu)
\end{equation}
denote the diagonal part of $R(t,\nu)$. Then we find a matrix-valued function
$N(t,\nu)$ satisfying
\begin{equation}
    \int_1^\infty \| N(t,\nu) \|^\sigma \frac{\d t}{t} < C'
\end{equation}
uniformly in $\nu\in\Upsilon$ such that the differential expression
\begin{multline}\label{eq:HW-com-eq}
    \big( t\partial_t  - D(t,\nu) - R(t,\nu)\big) \big( \mathrm I+N(t,\nu) \big)\\
    - \big(\mathrm I+N(t,\nu) \big)    \big( t\partial_t  - D(t,\nu) - F(t,\nu)\big)
    = B(t,\nu)
\end{multline}
satisfies
\begin{equation}\label{eq:HWcond'}
   \int_1^\infty \| B(t,\nu)\|^{\max\{\sigma/2,1\}} \frac{\d t}{t} <\infty.
\end{equation}
Furthermore, $N(t,\nu)\to0$ as $t\to\infty$ such that the matrix $\mathrm I+N(t,\nu)$ is invertible for $t\ge t_0$. Hence, $\tilde V = (\mathrm I+ N(t,\nu))^{-1} V$ solves the transformed problem 
\begin{equation}
    t\partial_t \tilde V = \big( D(t,\nu)+F(t,\nu) +  R_1(t,\nu) \big) \tilde V
\end{equation}
with $R_1(t,\nu) =  (\mathrm I+ N(t,\nu))^{-1} B(t,\nu)$ also satisfying \eqref{eq:HWcond'}.
\end{thm}
\begin{proof}
This follows \cite{Eastham} Section 1.5 and is a version of the diagonalisation scheme we applied earlier on. 
We set 
$ D_1 (t,\nu) = D(t,\nu) +F(t,\nu)$, $F(t,\nu)=\diag R(t,\nu) $
und denote $\tilde R(t,\nu) = R(t,\nu)-F(t,\nu)$. We construct $N(t,\nu)$ as solution to 
\begin{equation}\label{eq:HW-ce-n}
   t\partial_t N(t,\nu) = D(t,\nu) N(t,\nu) - N(t,\nu) D(t,\nu) +\tilde R(t,\nu),\qquad
   \lim_{t\to\infty} N(t,\nu) = 0,
\end{equation}
such that equation \eqref{eq:HW-com-eq} becomes
\begin{equation}\label{eq:B-def}
  B(t,\nu) = N(t,\nu) F(t,\nu)-R(t,\nu) N(t,\nu).
\end{equation}
In a first step we estimate $N(t,\nu)$. Considering individual matrix entries \eqref{eq:HW-ce-n} reads as
\begin{align}
   t\partial_t n_{jj} (t,\nu) &= 0,\\
   t\partial_t n_{ij}(t,\nu) &= (\mu_i(t,\nu)-\mu_j(t,\nu)) n_{ij}(t,\nu) + r_{ij}(t,\nu) 
\end{align}
such that the diagonal entries are given by $n_{jj}(t,\nu)=0$. For the off-diagonal entries we formulate integral representations and use the auxiliary function
\begin{equation}
   \delta_{ij}(t,\nu) = \int_{1}^t (\mu_i(s,\nu)-\mu_j(s,\nu)) \frac{\d s}{s}.
\end{equation}
Then the off-diagonal entries are given by Duhamel integrals
\begin{equation}
   n_{ij}(t,\nu) = - \e^{\delta_{ij}(t,\nu)} \int_t^\infty \e^{-\delta_{ij}(s,\nu)} r_{ij}(s,\nu)  \frac{\d s}{s} 
\end{equation}
for those $i,j$ where $\Re (\mu_i-\mu_j)\ge C_+>0$ and 
\begin{equation}
   n_{ij}(t,\nu) =  \e^{\delta_{ij}(t,\nu)} \int_{1}^t \e^{-\delta_{ij}(s,\nu)} r_{ij}(s,\nu)  \frac{\d s}{s} 
\end{equation}
for those with $\Re (\mu_i-\mu_j)\le  C_-<0$. It follows in particular that $n_{ij}(t,\nu)\to0$ as $t\to\infty$ and 
with $\pm C_\pm\ge\delta>0$ the estimates
\begin{equation}
   |n_{ij}(t,\nu)| \le \int_1^\infty s^{-\delta} |r_{ij}(ts^{\pm1},\nu)| \frac{\d s}{s},  
\end{equation}
the $\pm$-sign depending on the case of the Dichotomy condition. Therefore, the $L^\sigma$-property of $r_{ij}$ implies
by Minkowski inequality
\begin{equation}
  \left(\int_{1}^\infty |n_{ij}(t,\nu)|^\sigma \frac{\d t}{t}\right)^{1/\sigma} \le \int_1^\infty s^{-\delta} \left(\int_{1}^\infty |r_{ij}(ts^{\pm1},\nu)|^\sigma \frac{\d t}{t}\right)^{1/\sigma}\frac{\d s}{s},  
\end{equation}
and thus 
\begin{equation}
 \int_{1}^{\infty} \| N(t,\nu) \|^\sigma \frac{\d t}{t} <\infty.
\end{equation}
uniformly in $\nu\in\Upsilon$. Similarly, by H\"older's inequality and with $\sigma\sigma'=\sigma+\sigma'$.
\begin{align}
 \sup_t  |n_{ij}(t,\nu)| & \le \int_1^\infty s^{-\delta} |r_{ij}(ts^{\pm1},\nu)| \frac{\d s}{s}\notag\\
 & \le \left( \int_1^\infty s^{-\delta \sigma'} \frac{\d s}{s}\right)^{1/\sigma'} \left(\int_{1}^\infty |r_{ij}(ts^{\pm1},\nu)|^\sigma \frac{\d s}{s}\right)^{1/\sigma} 
\end{align}
uniformly in $\nu\in\Upsilon$. Hence, the matrix $N$ belongs to $L^{r}([1,\infty),\d t/t)$ for all $\sigma\le r\le \infty$ uniformly in $\nu$. If $\sigma\ge2$ then equation \eqref{eq:B-def} implies that $B(t,\xi)$ is product of two $L^\sigma$-functions and thus in $L^{\sigma/2}$. If $\sigma\in[1,2)$ then $\sigma'>\sigma$ and thus $B(t,\xi)$ is the product of an $L^\sigma$-function with an $L^{\sigma'}$-function and thus in $L^1$. 
\end{proof}

We distinguish two cases. If $\sigma\in(1,2]$ the transformation reduces the system to Levinson form and 
Theorem \ref{thm:Lev1} applies. If $\sigma$ is larger, than one application of the transform gives a new remainder satisfying \eqref{eq:HWcond} with $\sigma$ replaced by $\sigma/2$. 

In the first case the conclusion of the Theoren~\ref{HartWintner} is the existence of solutions 
\begin{equation}
    V_k(t,\nu) = \big( e_k + o(1) \big) t^{\mu_k} \exp \left(\int_0^t {r_{kk}(s,\nu)}\frac{\d s}{s}\right),\qquad k=1,\ldots,d,
\end{equation}
uniformly in the parameter, provided $D=\diag(\mu_1,\ldots,\mu_d)$ with distinct entries and $R\in L^\sigma([1,\infty),\d t/t)$.

\end{appendix}

\end{document}